\documentclass[oneside,reqno,12pt]{amsart}
\usepackage[greek,english]{babel}
\usepackage{fullpage}
\usepackage{amsthm}
\usepackage{amsbsy}
\usepackage{amsfonts}
\usepackage{graphicx}
\usepackage{hyperref}
\hypersetup{colorlinks=true, citecolor=blue, linkcolor=red}

\newtheorem{thm}{Theorem}
\newtheorem{cor}{Corollary}
\newtheorem{lem}{Lemma}
\newtheorem{pro}{Proposition}
\newtheorem{rem}{Remark}
\newtheorem{definition}{Definition}
\numberwithin{equation}{section} \numberwithin{lem}{section}
\numberwithin{thm}{section} \numberwithin{cor}{section}
\numberwithin{pro}{section} \numberwithin{rem}{section}

\def\ep{\varepsilon}
\def\eps{\varepsilon}

\def\R{\mathbb{R}}

\def\C{\mathbb{C}}

\begin{document}
\title{Interface layer of a two-component Bose-Einstein condensate}
\author{Amandine Aftalion}
\address{CNRS UMR 8100, Laboratoire de Math\'ematiques de Versailles, Universit\'{e} de Versailles Saint-Quentin, 45
avenue des Etats-Unis, 78035 Versailles Cedex, France.}
\email{amandine.aftalion@uvsq.fr}
\author{Christos Sourdis}
\address{Department of Mathematics,  University of
Turin,   Via Carlo Alberto 10,
20123, Turin, Italy.}
\email{christos.sourdis@unito.it}

\date{\today}
\begin{abstract}
This paper deals with the study of the behaviour of the wave functions of a two-component
 Bose-Einstein condensate
  near the interface, in the case of strong segregation. This
  yields a system of two coupled ODE's for which we want to have estimates
  on the asymptotic behaviour, as the strength of the coupling tends to infinity.
  %, of heteroclinic solutions (''domain walls").
  % that connect two equilibria.
   As in phase separation models, the leading order profile is a hyperbolic tangent. We construct an % inner and outer solution that we match at a point using the fact that  the Hamiltonian
 % of the system is constant. Then,
  approximate solution and use the properties of the associated linearized operator to perturb it into a genuine solution for which we have
  an asymptotic expansion. We prove that the constructed heteroclinic solutions are linearly nondegenerate, in the natural sense, and that there is
  a spectral gap,  independent of the large interaction parameter, between the zero eigenvalue (due to translations) at the bottom of the spectrum and the rest of the spectrum. Moreover, we prove a uniqueness result which implies that, in fact, the constructed heteroclinic is the unique minimizer (modulo translations) of the associated energy, for which we provide an expansion.
  %In particular,  this settles a conjecture raised recently by Alama, Bronsard, Contreras and Pelinovsky [Arch. Rational Mech. Anal. 215 (2015), 579--610]. Lastly, putting everything together, we are able to give an asymptotic expression for the minimal energy which is in agreement with that predicted in the physics literature.
  %Moreover, we establish a uniqueness result for such solutions.
\end{abstract}
 \maketitle

%\tableofcontents

\section{Introduction}\label{secIntro}%\input{introaacs.tex}
%%%%%%%%%%%%%%%%%%%%%%%%%%%%%%%%%%%%%%%%%%%%%%%%%%%%%%%%%%%%%%%%%%%%%%%%%%%%%%%%%%%%%%%%%%%%%%%%%%%
\subsection{The problem}
A two-component condensate is described by two complex valued wave
 functions  minimizing a Gross-Pitaevskii energy with a coupling term.
  According to the magnitude of the coupling parameter, the components can either
  coexist or segregate.

  The segregation behaviour in two-component condensates has been widely studied in the mathematics literature: regularity of the wave function \cite{ctv3,dancer2011,NoTaTeVe,sz2,sz,wang2015uniform,WeWe1}, regularity of the limiting interface \cite{CaffLin2,tavares2012regularity,zhang2015singularities}, asymptotic behaviour near the interface \cite{berestycki-wei2012,berestycki2}, $\Gamma$-convergence in the case of a trapped problem \cite{AL,goldman2015phase,GL}.

 This paper deals with the case of segregation, and more precisely, the study of the behaviour of the wave functions
  near the interface.  In the physics literature, there is a formal analysis of this small coexistence region which,  at leading order, is
   given by a hyperbolic tangent \cite{AO,barankov,vaninterface}. Here, we want to derive
  a rigorous asymptotic expansion of this transition layer, which will be useful in the analysis
  of more complex patterns.

The aim of this paper is therefore to study the positive solutions of the system \begin{equation}\label{eqEqGen}
	\left\{
	\begin{array}{c}
	-v_1''+v_1^3- v_1+\Lambda v_2^2 v_1=0, \\
	\\
	-v_2''+v_2^3-v_2+\Lambda v_1^2 v_2=0,
	\end{array}
	\right.
	\end{equation}
	\begin{equation}\label{eqBdryGen}
	(v_1,v_2)\to \left(0,1\right)\ \textrm{as}\ z\to -\infty,\ \ (v_1,v_2)\to
	\left(1,0\right)\ \textrm{as}\ z\to  +\infty.
	\end{equation} The segregation case corresponds to
	  \begin{equation}\label{eqHessian}\Lambda>1\end{equation} and the  limit
 $\Lambda \to \infty$.

The Hamiltonian
	\begin{equation}\label{eqGenHam}
	H=\sum_{i=1}^{2}\left[\frac{1}{2}(v_i')^2
	-\frac{\left(1-v_i^2\right)^2}{4}\right]-\frac{\Lambda}{2}v_1^2v_2^2,
	\end{equation}
	is constant along solutions of (\ref{eqEqGen})-(\ref{eqBdryGen}) and is equal to \begin{equation}\label{eqHamilton}
H=-\frac{1}{4}.
\end{equation}

It is known that  solutions of (\ref{eqEqGen})-(\ref{eqBdryGen}) are uniformly bounded independently of $\Lambda>1$ (see (\ref{eqAlama}) below).
Hence, by the general theory developed in  \cite{sz} and the references therein, they are uniformly Lipschitz continuous and converge, uniformly  as $\Lambda\to \infty$, to the merely Lipschitz continuous pair $\left(\chi_{(0,\infty)}U_1,\chi_{(-\infty,0)}U_2 \right)$, where
$U_1$ and $U_2$ denote
the unique solutions  respectively
of the following problems:
\begin{equation}\label{eqW1Gen}
	u''+ u-u^3=0,\ z>0;\ u(0)=0,\ u(z)\to 1  \ \textrm{as}\ z\to  +\infty,
\end{equation}
\begin{equation}\label{eqW2Gen}
	u''+ u-u^3=0,\ z<0;\ u(z)\to 1 \ \textrm{as}\ z\to  -\infty,\ u(0)=0,
\end{equation}
($\chi_I$ stands for the characteristic function of a set $I$).
In fact, we have the explicit formulas:
\[
U_i(z)=\tanh\left(\frac{z}{\sqrt{2}}\right),\ \ (-1)^iz\leq 0,\ \ i=1,2.
\]
A crucial observation   is that  the Hamiltonian structure of (\ref{eqW1Gen}), (\ref{eqW2Gen}) implies  the reflection property
\begin{equation}\label{eqGenPsi0}
	U_1'(0)+U_2'(0)=0, \ \
	\textrm{with}\ \
	\psi_0=U_1'(0)=\frac{1}{\sqrt{2}}.
\end{equation}
Of course, this follows at once from the explicit representations of $U_1$, $U_2$ but we would like to start convincing the reader that the specific form of the nonlinearity in (\ref{eqW1Gen}), (\ref{eqW2Gen}) is not of essential importance in the proofs.
On the one hand, the functions $\chi_{(0,\infty)}U_1,\chi_{(-\infty,0)}U_2$ do satisfy (\ref{eqEqGen})-(\ref{eqBdryGen}) for $z\neq 0$. On the other hand, their second derivatives blow-up at the origin as  delta masses.
To remedy this, guided by formal matched asymptotics and (\ref{eqGenPsi0}) (see also \cite{vaninterface}), we will instead use near the origin an approximate solution with  leading term
\begin{equation}\label{eqInprelimgh}
\left(\Lambda^{-\frac{1}{4}}V_1\left(\Lambda^{\frac{1}{4}}z
\right), \ \Lambda^{-\frac{1}{4}}V_2\left(\Lambda^{\frac{1}{4}}z
\right)\right),
%\ \ \textrm{for}\ \ |z|\leq (\ln \Lambda)\Lambda^{-\frac{1}{4}},
\end{equation}
where the pair $(V_1,V_2)$ is provided by the following proposition.
 \begin{pro}\label{proBerestycki}\cite{berestycki-wei2012,berestycki2}
There exists a unique solution $(V_1,V_2)$ with positive
components to the system
\begin{equation}\label{eqBUsystem}
\left\{\begin{array}{c}
  V''_1 =V_2^2 V_1, \\
    \\
  V''_2 =V_1^2 V_2, \\
\end{array} \right.
\end{equation}
 such that
\begin{equation}\label{eqBUasymp}
\frac{V_1}{x}\to \psi_0 \ \ \textrm{and}\ \ V_2\to 0\
\textrm{as}\ x\to +\infty,\end{equation} where $\psi_0>0$ is as in
(\ref{eqGenPsi0}), and
\begin{equation}\label{eqV1V2sym}
V_1(-x)=V_2(x), \ x\in \mathbb{R}.\end{equation}
Moreover,
\begin{equation}\label{eqV1V2asympt}
V_1(x)=\psi_0x+\kappa+\mathcal{O}\left(e^{-cx^2} \right)\
\textrm{and}\ V_2(x)=\mathcal{O}\left(e^{-cx^2} \right)\
\textrm{as}\ x \to +\infty,
\end{equation}
for some
$
\kappa\geq 0,
$
and these relations can be differentiated arbitrarily many times.
Every other entire solution of (\ref{eqBUsystem})
with positive components is given by
\begin{equation}\label{eqsymBU}
\left(\mu V_1\left(\mu (x-h) \right),\mu V_2\left(\mu (x-h)
\right) \right)
\end{equation}
for some $\mu>0$ and $h\in \mathbb{R}$.
\end{pro}

We emphasize that we will not need  the  above symmetry and uniqueness properties of the \emph{blow-up} profiles,
which were shown in
\cite{berestycki-wei2012} and \cite{berestycki2}  respectively by a nontrivial sliding method.
%Actually, we will end up with new (but indirect) proofs of them  (see Remark \ref{RemUnix}).

Our goal is to refine the  outer and inner approximate solutions in (\ref{eqW1Gen})-(\ref{eqW2Gen}) and (\ref{eqInprelimgh}) respectively, carefully glue them together and show that the resulting global approximate solution can be perturbed to a genuine one.

\subsection{Main results}
The main result of the paper is the following:
\begin{thm}\label{thmMain}
If $\Lambda>0$ is sufficiently large, problem
(\ref{eqEqGen})-(\ref{eqBdryGen}) has a solution
$(v_{1,\Lambda},v_{2,\Lambda})$ such
that
\begin{eqnarray}
\label{eqmonotTHM} v_{1,\Lambda}'(z)>0, \ \ v_{2,\Lambda}'(z)<0\ \ \textrm{for}\ \   z\in \mathbb{R},
\\ \label{v1U}
v_{i,\Lambda}(z)=U_i\left(z-(-1)^i\psi_0^{-1}\kappa
 \Lambda^{-\frac{1}{4}}\right)+\mathcal{O}\left((\ln \Lambda)\Lambda^{-\frac{3}{4}} \right)e^{-c|z|},\\  \textrm{uniformly for}\ \ (-1)^{i+1}z\geq (\ln \Lambda)\Lambda^{-\frac{1}{4}},\ \ \textrm{as}\ \ \Lambda\to \infty,  \nonumber  \\ \label{v1out}
v_{i,\Lambda}(z)=\Lambda^{-\frac{1}{4}}V_i\left(\Lambda^{\frac{1}{4}}z
\right)+\mathcal{O}\left(\Lambda^{-\frac{3}{4}}+|z|^3\right),
\end{eqnarray}
uniformly on $\left[-(\ln \Lambda)\Lambda^{-\frac{1}{4}}, (\ln
\Lambda)\Lambda^{-\frac{1}{4}} \right]$, as $\Lambda \to \infty$, $i=1,2$, where $U_1$, $U_2$ are the unique solutions of (\ref{eqW1Gen}), (\ref{eqW2Gen}) respectively,   $(V_1,V_2)$ is the solution
 of (\ref{eqBUsystem})-(\ref{eqBUasymp})-(\ref{eqV1V2sym}), and $\kappa>0$ is as in (\ref{eqV1V2asympt}).
Furthermore, for any $m>0$, we have
\begin{equation}\label{eqEmvoRRR}
v_{i,\Lambda}(z)\leq C \Lambda^{-\frac{1}{4}} e^{-c\Lambda^{\frac{1}{2}}z^2}+\mathcal{O}(\Lambda^{-m}),
\ \ (-1)^i z\in \left[\Lambda^{-\frac{1}{4}}, (\ln \Lambda)\Lambda^{-\frac{1}{4}}\right],\end{equation}
and
\begin{equation}\label{v1expdecay}
  v_{i,\Lambda}(z) \leq \Lambda^{-m}e^{-c(\ln \Lambda)^\frac{1}{2}\Lambda^\frac{1}{4}|z| },  \ \ (-1)^iz\geq (\ln \Lambda)\Lambda^{-\frac{1}{4}}, \ \ i=1,2,
\end{equation}
as $\Lambda \to \infty$.
Moreover,
\begin{eqnarray}\label{v1UGrad}
v_{i,\Lambda}'(z)=U_i'\left(z-(-1)^i\psi_0^{-1}\kappa
 \Lambda^{-\frac{1}{4}}\right)+\mathcal{O}\left((\ln \Lambda)\Lambda^{-\frac{3}{4}} \right)\left(|z|+\Lambda^{-\frac{1}{4}} \right)e^{-c|z|},\\  \textrm{uniformly for}\ \ (-1)^{i+1}z\geq (\ln \Lambda)\Lambda^{-\frac{1}{4}},\ \ \textrm{as}\ \ \Lambda\to \infty,\nonumber \\ \label{v1outt}
v'_{i,\Lambda}(z)={V_i'}\left(\Lambda^{\frac{1}{4}}z
\right)+\mathcal{O}\left(\Lambda^{-\frac{1}{2}}+|z|^2\right),
\end{eqnarray}
uniformly on $\left[-(\ln \Lambda)\Lambda^{-\frac{1}{4}}, (\ln
\Lambda)\Lambda^{-\frac{1}{4}} \right]$, as $\Lambda \to \infty$, $i=1,2$.
\end{thm}

We can also show that $(v_{1,\Lambda},v_{2,\Lambda})$  is nondegenerate, in the natural sense,  and that there is a spectral gap.

\begin{thm}\label{thmNonDeg}
	Let $(v_1,v_2)$ be the heteroclinic solution to (\ref{eqEqGen})-(\ref{eqBdryGen}) which is constructed in Theorem \ref{thmMain}.
	Then, if $\Lambda>0$ is sufficiently large, the spectrum of the linearized operator
	\begin{equation}\label{eqM}
	\textbf{M} \left(\begin{array}{c}
	\varphi_1 \\
	\\
	\varphi_2 \\
	\end{array} \right)=\left(\begin{array}{c}
	-\varphi_1''+(3v_1^2-1)\varphi_1+\Lambda v_2^2 \varphi_1+2\Lambda v_1v_2\varphi_2 \\
	\\
	-\varphi_2''+(3v_2^2-1)\varphi_2+\Lambda v_1^2 \varphi_2+2\Lambda v_1v_2\varphi_1 \\
	\end{array} \right),
	\end{equation}
	in $L^2(\mathbb{R})\times L^2(\mathbb{R})$ with domain $H^2(\mathbb{R})\times H^2(\mathbb{R})$ is structured as follows:
	\begin{itemize}
		\item $0$ is the first eigenvalue and has $(v_1',v_2')$ as the associated eigenfunction,
		\item the rest of the spectrum is  contained in $(c,\infty)$ for some $c>0$.
	\end{itemize}
\end{thm}

Furthermore, we have the following uniqueness result.

\begin{thm}\label{thmUniqNonSym}
	If $\Lambda$ is as in (\ref{eqHessian}), there exists at most one solution (modulo translations) to problem (\ref{eqEqGen})-(\ref{eqBdryGen}) with positive components such that one of them is strictly monotone.
\end{thm}

An important consequence of the above theorem is that  the minimization problem
\begin{equation}\label{eqsigmaL}
\sigma_\Lambda=\inf_{(v_1,v_2)\in \mathcal{Y}}E_\Lambda(v_1,v_2),
\end{equation}
where
\begin{equation}\label{eqEnergyComplexV}
E_\Lambda(v_1,v_2)=\int_{\mathbb{R}}^{}\left\{\sum_{i=1}^{2}\left[\frac{1}{2}(v_i')^2
	+\frac{\left(1-v_i^2\right)^2}{4}\right]+\frac{\Lambda}{2}v_1^2v_2^2-\frac{1}{4}\right\}dz,
\end{equation}
and
\begin{equation}\label{eqY}
\mathcal{Y}=\left\{(v_1,v_2)\in H_{loc}^1(\mathbb{R})\times H_{loc}^1(\mathbb{R})\ \ \textrm{which\ sastisfy}\ \ (\ref{eqBdryGen})  \right\}.
\end{equation} has a unique solution
(modulo translations), which is that of Theorem \ref{thmMain}. In fact, this settles a conjecture from \cite[Sec. 5]{alamaARMA15}, in relation to
the stability properties of the family of minimizers of the  complex valued version of (\ref{eqEnergyComplexV}) with  respect to the associated nonlinear Schr\"{o}dinger dynamics. Armed with the estimates provided by our main theorem, we can give an asymptotic expression for its minimal energy $\sigma_\Lambda$:
\begin{cor}\label{corenergyExp}
As $\Lambda \to \infty$,
\[
\sigma_\Lambda=\frac{2\sqrt{2}}{3}+2\Lambda^{-\frac{1}{4}}\int_{-\infty}^{\infty}V'_1\left(V'_1-\psi_0\right)dx+\mathcal{O}\left( (\ln \Lambda)^3 \Lambda^{-\frac{3}{4}}\right).
\]
\end{cor}
The minimal energy $\sigma_\Lambda$ in (\ref{eqsigmaL}) represents the interface tension of the condensate. In \cite{vaninterface,mishonov},  a formal series  expansion of $\sigma_\Lambda$ in powers of $\Lambda^{-\frac{1}{4}}$ is given by using matched asymptotic analysis.
The
first term of this series was recovered rigorously in \cite{goldman2015phase} via variational arguments. In comparison, our Corollary \ref{corenergyExp} recovers rigorously the first three terms
of that prediction together with the correct order of the fourth term, modulo the 'artificial' logarithmic factor.

\subsection{Main steps of the proofs}
The idea of the proof of Theorem \ref{thmMain} is to construct an approximate solution and then to perturb it into a genuine solution, using the linearized operator.
This approach has been  extensively pursued
and thoroughly developed in the past years for constructing localized solutions to
elliptic problems, mainly involving spike-transition
layer  or bubbling phenomena.
However, as will be apparent, some important differences occur with respect to the standard technique.
This should already be expected  from
the irregular form of the  singular limit solution of the problem in hand. In particular, its corner layered
structure at the origin  forces the corresponding blow-up profiles  to be
unbounded, in sharp
contrast to the situation in the aforementioned widely studied concentration problems.

 Firstly, it is not difficult to construct an outer (exact) solution of (\ref{eqEqGen})-(\ref{eqBdryGen}) for $|z|\geq ( \ln \Lambda) \Lambda^{-1/4}$, which  satisfies
the expected asymptotic behaviour (\ref{v1U}). We do not prescribe conditions at the end points, as these will be controlled by the inner solution that we construct next. We point out that this construction is possible by the nondegeneracy of the solutions $U_1$, $U_2$ of (\ref{eqW1Gen}), (\ref{eqW2Gen}) respectively (see Lemma \ref{lemUnondeg} below). Then, we construct an inner solution for $|z|\leq (\ln
\Lambda)\Lambda^{-\frac{1}{4}}$
based on the blow-up profile $(V_1,V_2)$ of Proposition \ref{proBerestycki} and on its nondegeneracy  (in the natural sense, see Proposition \ref{proBerestNondegen} below). Actually, by exploiting the scaling invariance of (\ref{eqBUsystem}), illustrated by the parameter $\mu$
in (\ref{eqsymBU}), we can define a one-parameter family of such inner solutions. Let us note that there is no gain in exploiting the translation invariance of
(\ref{eqBUsystem}) for this purpose, as the whole problem (\ref{eqEqGen})-(\ref{eqBdryGen}) is itself translation invariant.
We emphasize that the construction of the inner solution relies heavily on the study of the linearization of
the blow-up problem (\ref{eqBUsystem}) (see Subsection \ref{subsecInner}), which could also prove useful in other
settings.
 More precisely,  the linearized operator of the blow-up system (\ref{eqBUsystem}) at $(V_1,V_2)$
contains an element $(E_1,E_2)$ with linear growth in its (formal) kernel,  due to the aforementioned invariance of (\ref{eqBUsystem}) under scaling. We can use a constant multiple of this element as the parameter in the  inner solutions.
In order to efficiently glue together the inner and outer approximations, we need to adjust the constant parameters involved in their separate constructions. For technical reasons, which will be clear from the proofs, instead of matching these approximations in the $C^1$-sense over an intermediate zone, we  match them continuously only at the points $\pm( \ln \Lambda) \Lambda^{-1/4}$.
At first sight, this unconventional argument  might look
rather counterintuitive, as it would create  jumps on the gradients at the gluing points. But, by choosing the free parameter in the inner solution so that  the
 Hamiltonian has the same value on each side, it turns out that these jumps on the gradient are actually transcendentally small.
 The resulting global approximation fails to be an exact solution to the problem by essentially just a transcendentally small factor. Naturally, the first thing that comes to mind is to try  to perturb it to a genuine solution by  some type of local inversion argument, through the study of the associated linearized operator about it.
 %In turn, this implies at once the existence of a sufficiently smooth global approximate solution to the problem which leaves a tra
%We perturb our approximate solution further using the kernel element $(E_1,E_2)$
%to get a one parameter family of solutions keeping the Hamiltonian at its expected
%value.
% The last step is to show that this approximate solution can be perturbed to
%a genuine solution for large $\Lambda$.
However, the translation invariance of (\ref{eqEqGen})-(\ref{eqBdryGen}) implies that the latter operator
is nearly non-invertible, as the derivative of the approximate solution fails to be in its kernel
by at most a transcendentally small factor. Nevertheless, by combining the linear analysis that we developed separately
for the inner and outer problems, we can show that the global linearized operator does not have other
elements in its spectrum tending to zero, as $\Lambda \to \infty$, besides  the transcendentally small eigenvalue
 at the bottom. In fact, the latter eigenvalue turns out to be simple (see Theorem \ref{thmNonDeg} and Proposition \ref{proGenLinearQ}). Consequently, we are led to use a Lyapunov-Schmidt variational reduction method for the perturbation argument.

The proof of our uniqueness result rests upon a homotopy argument, taking advantage of the nondegeneracy property of this type of monotone solutions to (\ref{eqEqGen})-(\ref{eqBdryGen}), which holds for $\Lambda$ in the range (\ref{eqHessian}).

\subsection{Physical motivation and known results}

 A rotating two-component Bose-Einstein condensate is described by the ground state of the following energy
\begin{equation}\label{eqEnergyOmega}
\begin{split}
\mathcal{E}(u_1,u_2)=\sum_{j=1}^2 \int_{\R^2} \left\{ \frac{|\nabla u_j|^2}{2} + \frac{V(|x|)}{2\ep^2}|u_j|^2 +\frac{g_j}{4\ep^2}|u_j|^4 -\Omega x^{\perp}\cdot(i u_j,\nabla u_j) \right\} \,dx \\
+\frac{g}{2\ep^2} \int_{\R^2} |u_1|^2|u_2|^2 \,dx
\end{split}
\end{equation}
 in the set
\begin{equation}\label{H_space_definition}
\mathcal{H}=\left\{(u_1,u_2):\ u_j\in H^1(\R^2;\C), \
\int_{\R^2}V(|x|)|u_j|^2\,dx<\infty, \ \|u_j\|_{L^2(\R^2)}=1,\
j=1,2\right\}.
\end{equation} The trapping potential $V(|x|)$ is usually taken to be $|x|^2$, corresponding to the experiments.
The parameters $g_1,g_2,g,\ep$ and $\Omega$ are positive:  $g_j$ is the self-interaction of each component (intracomponent coupling) while $g$ measures the effect of interaction between the two components (intercomponent coupling); $\Omega$ is the
angular velocity corresponding to the rotation of the condensate, $x^\perp=(-x_2,x_1)$ and $\cdot$ is the
scalar product for vectors, whereas $(\ ,\ )$ is the complex scalar product, so
that  we have
\[
x^{\perp}\cdot(iu,\nabla u)=x^{\perp}\cdot\frac{iu\nabla\bar u-i\bar
u\nabla u}{2}= -x_2\frac{iu\partial_{x_1}\bar u-i\bar
u\partial_{x_1}u}{2}+x_1\frac{iu\partial_{x_2}\bar u-i\bar
u\partial_{x_2}u}{2}.
\]
 The existence and behaviour of the
minimizers in the limit when $\ep$ is
small, describing strong interactions, is also called the Thomas-Fermi limit. Even though  the
 interaction is only through the modulus, it can produce effects on the phases of each component, and in particular
  on the singularities or vortices. A full phase diagram has been computed in \cite{AM}.

  If the condition $
g^2<g_1g_2$ is satisfied, it means that
 the two components $u_1$ and $u_2$ of the
minimizers can coexist, as opposed to the segregation case
$g^2>g_1g_2$. This is discussed and explained in \cite{ANS,AM}. The ground state at $\Omega=0$ in the
coexistence case has been studied in the small $\eps$ limit in \cite{ANS,gallo}.
 
 In the present paper, we are interested
 in the segregation case. The $\Gamma$ limit of (\ref{eqEnergyOmega})-(\ref{H_space_definition}) with $\Omega=0$ in the case where $g^2>g_1g_2$, $g_1=g_2$ and $g=g_\varepsilon \to \infty$
   has been studied in \cite{AL}. A change of functions
   is used, namely $(v,\varphi)$, where
   \begin{equation}\label{eqDefRad}
   v^2=u_1^2+u_2^2\hbox{ and } \cos\varphi= \frac {u_1^2-u_2^2}{u_1^2+u_2^2}.\end{equation}
   A $\Gamma$ limit is obtained on the functional for $(v,\varphi)$.
 The limiting problem is given by two domains $D_1$ (for component 1) and $D_2$ (for component 2) for
 which the interface minimizes a perimeter type problem weighted by the trapping potential.
 
 As conjectured in \cite{berestycki-wei2012}, the minimum of $v$ satisfies \[c \Lambda ^{-\frac{1}{4}}\leq \inf_{\mathbb{R}}v \leq C \Lambda^{-\frac{1}{4}},\]
for some constants $c,C>0$  independent of large $\Lambda$. Actually, an analogous estimate was established very recently in \cite{sz2,wang2015uniform} for uniformly bounded solutions to a broad class of elliptic systems. Clearly, the above estimate is a particular consequence of our Theorem \ref{thmMain} and the comments preceding Corollary \ref{corenergyExp}.
 
   Further $\Gamma$ convergence results for  $g=g_\eps$ fixed have been proved by \cite{GL}, and extended for $g_1 \neq g_2$ by \cite{goldman2015phase}. In these two papers, near the  interface, the interaction between the two components of the condensate is  governed by  the system
 \begin{eqnarray*}&&- u_1'' +g_1 u_1^2 u_1+g u_2^2 u_1=\lambda_1 u_1,\\
 &&- u_2'' +g_2 u_2^2 u_2+g u_1^2 u_2=\lambda_2 u_2,
 \end{eqnarray*}
for some constants $\lambda_1,\lambda_2>0$ corresponding to the Lagrange multipliers of (\ref{eqEnergyOmega})-(\ref{H_space_definition}). In the case where $g=\Lambda$ and $g_i,\lambda_i$ are  equal to 1, this gives rise to our one dimensional problem and they prove that the $\Gamma$ limit of the full 2D problem in a bounded domain is given by the minimization of $\sigma_\Lambda$.

\begin{rem}
In (\ref{eqEqGen}), we have taken all the constants in front of the non-coupled terms
to be equal to one. However, all of our arguments carry over easily to the positive solutions of
\begin{equation}\label{eqEqGenOld}
	\left\{
	\begin{array}{c}
	-v_1''+g_1v_1^3-\lambda_1 v_1+\Lambda v_2^2 v_1=0, \\
	\\
	-\nu v_2''+g_2v_2^3-\lambda_2v_2+\Lambda v_1^2 v_2=0,
	\end{array}
	\right.
	\end{equation}
	\begin{equation}\label{eqBdryGenOld}
	(v_1,v_2)\to \left(0,\sqrt{\frac{\lambda_2}{g_2}}\right)\ \textrm{as}\ z\to -\infty,\ \ (v_1,v_2)\to
	\left(\sqrt{\frac{\lambda_1}{g_1}},0\right)\ \textrm{as}\ z\to  +\infty,
	\end{equation}
	 for values of the parameter \begin{equation}\label{eqHessian222}\Lambda>\sqrt{g_1g_2},\end{equation} assuming that $\nu>0$, and the positive constants $g_1,g_2,\lambda_1,\lambda_2$ satisfy
	\begin{equation}\label{eqEqPres}
	\frac{\lambda_1^2}{g_1}=\frac{\lambda_2^2}{g_2}.
	\end{equation}
	 Thanks to (\ref{eqEqPres}), all the constants $g_i,\lambda_i$, $i=1,2,$ and $\nu$  can be scaled out since a solution is given by
\[
\sqrt{\frac{g_1}{\lambda_1}}v_1\left(\sqrt{\frac{1}{\lambda_1}}z \right),\ \ \sqrt{\frac{g_2}{\lambda_2}}v_2\left(\sqrt{\frac{\nu}{\lambda_2}}z \right),
		\]
		where $v_1,v_2$ satisfy (\ref{eqEqGen})-(\ref{eqBdryGen}), while the new coupling constant is $\tilde{\Lambda}=\frac{\lambda_2 \Lambda}{\lambda_1 g_2}$.
We point out that (\ref{eqEqPres}) is a necessary condition for the existence of solutions to (\ref{eqEqGenOld})-(\ref{eqBdryGenOld}) since the corresponding Hamiltonian is conserved and therefore the limit at $+/-\infty$ has to be the same.

Conversely, by means of variational arguments, it was shown recently in \cite{alamaARMA15} and \cite{goldman2015phase} that condition (\ref{eqEqPres}) is also sufficient (see  also \cite{alikakosFuscoIndiana,sternbergHeteroclinic2016}) for the existence of solutions.
		
\end{rem}

\hfill

The basic interaction of a two component condensate is through a modulus term, but
 other interactions include a Rabi coupling or a spin orbit coupling.
 The precise knowledge of the interface behaviour will prove useful to analyze more
    complicated patterns: \begin{itemize} \item the segregation case in the spin orbit coupling where there are vortex sheets \cite{KT},
     \item in the case of Rabi coupling, a vortex and anti-vortex pair create a vortex molecule, where two
     vortices are eventually connected by a domain wall of relative phase \cite{KTU,son},
     \item
    half vortices, where a vortex in one component corresponds to a peak
    in the other component \cite{catelani}.\end{itemize}
    In order to analyze the singularity patterns in all these cases, one needs to make an energy expansion
    in order to determine the energy of the specific configuration. Because of the transition layer between
    the two species, one needs to have a precise estimate of the decrease of the modulus of the wave
    functions, which is precisely given by our main theorem. Therefore,
    what we prove in Theorems \ref{thmMain}, \ref{thmNonDeg}
     is expected to be  extremely useful for the construction of upper bounds for these patterns.
  %\subsection{Organization of the paper}

\subsection{Outline of the paper} In Section \ref{secApprox}, we construct our approximate solution to the problem, in the outer and  inner regions separately. Then, in Section \ref{secMatch} we adjust them further so that they match conveniently for our purposes. In Sections \ref{secSolInner} and \ref{secSolOutReal}, we perturb the resulting inner and outer approximations respectively to   genuine ones .  In Section \ref{secMainRes}, we prove Theorems \ref{thmMain} and \ref{thmNonDeg}. In Section \ref{secUniq} we prove Theorem \ref{thmUniqNonSym}. Finally, in Section \ref{secEnergetic} we show Corollary \ref{corenergyExp}.
%\subsection{Notation}
%By $c,\ C$ we will denote small, large positive generic  constants that are independent of large $\Lambda>0$. The value of $\Lambda>0$ will  increase from line to line so that all the previous relations hold.
\subsection{Notation}\label{subsecNotation}
By $c/C$ we will denote small/large positive generic  constants that are independent of large $\Lambda>0$ and whose value will decrease/increase
as the paper moves on. The value of $\Lambda>0$ will also increase from line to line so that all the previous relations hold. According  to the Landau notation, a number $\rho$ will be of order $\mathcal{O}(\Lambda^{-m})$
as $\Lambda \to \infty$, for some $m\in \mathbb{R}$, if $|\rho|\leq C \Lambda^{-m}$ for $\Lambda$ sufficiently large; a number $\rho$ will be of order $o(\Lambda^{-m})$ as $\Lambda \to \infty$, for some $m \in \mathbb{R}$, if $\Lambda^{m}\rho \to 0 $ as $\Lambda \to \infty$;  a number $\rho$ will be of order $\mathcal{O}(\Lambda^{-\infty })$ as $\Lambda \to \infty$  if $\rho =\mathcal{O}(\Lambda^{-m})$, for any $m>1$, as $\Lambda \to \infty$. We will remove the obvious dependence on $\Lambda$ of various
functions.

\section{The approximate solution}
% $(v_{1,ap},v_{2,ap})$}
\label{secApprox} In this section, we will construct a sufficiently
good approximate solution to  problem
(\ref{eqEqGen})-(\ref{eqBdryGen}) for large $\Lambda>0$.
\subsection{The outer solution $(v_{1,out},v_{2,out})$}% $(v_{1,out},v_{2,out})$}
\label{subsecOuter} In this subsection, we will construct an
approximate solution to problem (\ref{eqEqGen})-(\ref{eqBdryGen})
 except at the origin, where it
looses its smoothness.
\subsubsection{The outer profiles $U_1,\ U_2$}\label{subsubsecU}
The building blocks of this construction will be the unique
solutions $U_1, U_2$ of problems (\ref{eqW1Gen}), (\ref{eqW2Gen}) respectively. Actually, we will restrict our attention to $U_1$, as the corresponding
analysis for $U_2$ is completely analogous.
For future reference, let us note that
\begin{equation}\label{eqUmonot}
U'_1(z)>0, \ z\geq 0,
\end{equation}
and
\begin{equation}\label{eqUdec}
1-U_1(z)+U_1'(z)-U_1''(z)\leq Ce^{-cz},\ z\geq 0.
\end{equation}

We will also need the following  properties for the associated
linearized operator, which are well known and essentially follow
from (\ref{eqUmonot})-(\ref{eqUdec}).
\begin{lem}\label{lemUnondeg}
	Let $\phi \in C^2[0,\infty)$ be bounded and satisfy
	\begin{equation}\label{eqbistable}
	-\phi''+\left(3U_1^2(z)-1 \right)\phi=0,\ \ z>0.
	\end{equation}
%	for some constant $\nu\leq 0$.
 Then, we have that
	\[ \phi \equiv c U_1' \ \  \textrm{for\ some}\ \   c\in \mathbb{R}.
	\]
In particular, if $\phi(0)=0$, then $\phi \equiv 0$.
\end{lem}
\begin{proof} The desired assertions of the lemma follow immediately  from the observation that, besides of $U_1'$, the differential operator
	in the lefthand side of (\ref{eqbistable})  has also  an unbounded function in its two-dimensional kernel. Indeed,  otherwise the Wronskian would be zero,  by (\ref{eqUdec}) and the fact that bounded elements in the kernel also have bounded derivatives (by a standard interpolation argument, see for instance (\ref{eqInterpoNew}) below).
\end{proof}
%\begin{proof}
%Suppose that $\phi$ is non trivial.  Since $\left(3U^2(z)-1
%\right)\to 2$ as $z\to \infty$, by Sturm's theorem and without
%loss of generality, we may assume that there exists a $z_0\geq 0$
%such that $\phi(z_0)=0$ and $\phi(z)>0$ for $z>z_0$. Then, we can
%multiply the equation by $U'$, integrate by parts the result over
%$(z_0,\infty)$, and make use of (\ref{eqUmonot})-(\ref{eqUdec}) to
%find that $\phi'(z_0)\leq 0$. We have thus arrived at a
%contradiction, which completes the proof.
%\end{proof}
\subsubsection{The construction of the outer approximate
solution}\label{subsubsecOutDef} We can now define our outer approximate
solution as
\begin{equation}\label{eqOutDef}
v_{1,out}(z)=U_1(z+\xi_1)+\tau_1 U_1'(z+\xi_1),\ v_{2,out}(z)=0 \
\textrm{for}\ z\geq (\ln \Lambda)\Lambda^{-\frac{1}{4}},
\end{equation}
with
\begin{equation}\label{eqconstr}
\xi_1=\mathcal{O}(\Lambda^{-\frac{1}{4}})\ \textrm{and}\
\tau_1=\mathcal{O}\left((\ln \Lambda)\Lambda^{-\frac{3}{4}}\right)\ \textrm{as}\ \Lambda \to
\infty,
\end{equation}
to be determined. Analogously we define it for $z\leq -(\ln \Lambda)\Lambda^{-\frac{1}{4}}$. We point out that the choice of the power $1/4$ in (\ref{eqOutDef}) is motivated by a
formal blow-up analysis (see the next subsection), whereas the choice of powers in (\ref{eqconstr}) is an a-posteriori result of matching considerations (see Subsection \ref{subsecMatchCont} below).
\subsubsection{The remainder of the outer approximate solution}\label{subsubsecOutRem}
We note that $(v_{1,out},v_{2,out})$ satisfies the desired
asymptotic behaviour (\ref{eqBdryGen}) exactly, while it satisfies the
system (\ref{eqEqGen}) approximately as is shown in the next
lemma.
\begin{lem}\label{lemRemOut}
The remainder
\[
R(v_{1,out},v_{2,out})=\left(\begin{array}{c}
  -v_{1,out}''+v_{1,out}^3-v_{1,out}+\Lambda v_{2,out}^2 v_{1,out}  \\
    \\
  -v_{2,out}''+v_{2,out}^3-v_{2,out}+\Lambda v_{1,out}^2 v_{2,out}  \\
\end{array}
 \right)
\]
that is left by $(v_{1,out},v_{2,out})$ in (\ref{eqEqGen}) satisfies
\[
R(v_{1,out},v_{2,out})=\left(\begin{array}{c}
   \mathcal{O}\left((\ln \Lambda)^2\Lambda^{-\frac{3}{2}}\right)e^{-cz} \\
    \\
 0 \\
\end{array}
 \right),
\]
uniformly for $z\geq (\ln \Lambda)\Lambda^{-\frac{1}{4}}$, as $\Lambda \to \infty$, and an analogous estimate holds for $z\leq -(\ln \Lambda)\Lambda^{-\frac{1}{4}}$.
\end{lem}
\begin{proof}
Keeping in mind that $v_{2,out}$ is identically zero, we
 observe that,  by virtue of (\ref{eqW1Gen}), we have
\[
-v_{1,out}''+v_{1,out}^3-v_{1,out}=\tau_1^3\left(U_1'(z+\xi_1)
\right)^3+3\tau_1^2 U_1(z+\xi_1)\left(U_1'(z+\xi_1) \right)^2,
\]
and then use (\ref{eqUdec}), (\ref{eqconstr}). The proof for $z\leq -(\ln \Lambda)\Lambda^{-\frac{1}{4}}$ is analogous.
\end{proof}
\subsection{The inner approximate solution $(v_{1,in},v_{2,in})$} \label{subsecInner}
In this subsection, we will construct an approximate solution to the
 system (\ref{eqEqGen}) which, however, is effective only in a
small neighborhood of the origin. Nevertheless, it will have the
appropriate behaviour so as to be easily ``continued'' away from
the origin by the outer solution of the previous subsection.

\subsubsection{The blow-up profile $(V_1,V_2)$} \label{subsubsecV1V2}
Based on a formal blow-up analysis and the behaviour of the outer approximate
solution near the origin, the building blocks will be special solutions of a limiting
problem, described in Proposition \ref{proBerestycki} which is due to
\cite{berestycki-wei2012,berestycki2}.

Moreover, by the convexity of $V_1,V_2$, it follows easily that
\begin{equation}\label{eqV1V2monot}
V'_1>0,\ V'_2<0,\ x\in \mathbb{R}.
\end{equation}
Actually, it is not hard to show that $\kappa> 0$.
 Indeed, we observe that the auxiliary function
\[
R(x)=V_1(x)-\psi_0x-V_2(x),\ \ x\geq 0,
\]
satisfies
\[
{R''}=V_1V_2(V_2-V_1)<0,\ \ x>0,
\]
while
\[
R(0)=0,\ \ \lim_{x\to +\infty}R(x)=\kappa.
\]

The invariance   of system (\ref{eqBUsystem}) under translation
and scaling, described in (\ref{eqsymBU}), implies that the
associated linearized operator \begin{equation}\label{eqL} L\left(\begin{array}{c}
  \Phi_1 \\

  \Phi_2 \\
\end{array}
\right)=\left(\begin{array}{c}
  -\Phi''_1+V_2^2\Phi_1+2V_1V_2\Phi_2 \\

  -\Phi''_2+V_1^2\Phi_2+2V_1V_2\Phi_1 \\
\end{array}
\right)
\end{equation}
has
\begin{equation}\label{eqkernel}
(V'_1,V'_2)\ \ \textrm{and}\ \
(xV'_1+V_1,xV'_2+V_2)
\end{equation}
amongst  its four-dimensional kernel. The next proposition, also proven in \cite{berestycki-wei2012}, will play
a key role in what will follow.
\begin{pro}\label{proBerestNondegen}
If $\Phi_1, \Phi_2 \in C^2(\mathbb{R})\cap L^\infty(\mathbb{R})$ satisfy
\[
L\left(\begin{array}{c}
  \Phi_1 \\

  \Phi_2 \\
\end{array}
\right)=\left(\begin{array}{c}
  0 \\

  0 \\
\end{array}
\right),
\]
then
\[
(\Phi_1,\Phi_2)\equiv\lambda (V'_1,V'_2)
\]
for some $\lambda \in \mathbb{R}$.
\end{pro}

We emphasize that the proof of the above proposition is based on the monotonicity property (\ref{eqV1V2monot}). In particular, no use is made of the symmetry (\ref{eqsymBU}) or the uniqueness property of $(V_1,V_2)$. In fact, the latter properties are considerably harder to establish.

\subsubsection{Construction of the inner approximate
solution}\label{subsubsecvin} Motivated from the above and \cite{berestycki-wei2012}, we
consider the stretched variable
\begin{equation}\label{eqx}
x=\mu \Lambda^\frac{1}{4}z
\end{equation}
with
\begin{equation}\label{eqmuprelim}
\mu=1+\mathcal{O}(\Lambda^{-\frac{1}{2}})\ \textrm{as}\ \Lambda\to
\infty
\end{equation}
to be determined (the last relation is an a-posteriori consequence of matching considerations, see Subsection \ref{subsecMatchCont} below). Then, we seek an inner approximate solution
$(v_{1,in},v_{2,in})$ to (\ref{eqEqGen}) in the form
\begin{equation}\label{eqvinprelim}
v_{i,in}(z)=\mu \Lambda^{-\frac{1}{4}}V_i(x)+\Phi_i(x),\ \ |z|\leq
(\ln \Lambda)\Lambda^{-\frac{1}{4}}, \ \ i=1,2,
\end{equation}
where the functions $\Phi_1,\Phi_2$ are also to be determined. Actually, at first, we were tempted to also exploit the translation invariance of (\ref{eqBUsystem}) by introducing a shift parameter in the stretched variable $x$,
similarly to (\ref{eqsymBU}), but then realized that this is not needed as the  problem (\ref{eqEqGen})-(\ref{eqBdryGen}) itself is translation invariant.
Using (\ref{eqBUsystem}), we find that the remainder which is left
in the first equation of (\ref{eqEqGen}) by this approximation is
\begin{equation}\label{eqRemInBasic}
\begin{array}{c}
  -\mu^2\Lambda^\frac{1}{2}\Phi''_1+\mu^3\Lambda^{-\frac{3}{4}}V_1^3+\Phi_1^3+3\mu^2\Lambda^{-\frac{1}{2}}V_1^2\Phi_1
+3\mu \Lambda^{-\frac{1}{4}}V_1\Phi_1^2+\mu^2 \Lambda^\frac{1}{2}
V_2^2 \Phi_1
 \\
    \\
  +\mu \Lambda^\frac{3}{4}\Phi_2^2V_1+\Lambda\Phi_2^2 \Phi_1+2\mu^2
\Lambda^\frac{1}{2}V_1V_2 \Phi_2+2\mu \Lambda^\frac{3}{4}V_2
\Phi_1 \Phi_2-\mu \Lambda^{-\frac{1}{4}}V_1-\Phi_1,
 \\
\end{array}
\end{equation}
and an analogous relation holds for the second equation. Hence, we
would like for $(\Phi_1,\Phi_2)$ to satisfy
\[
L\left(\begin{array}{c}
  \Phi_1 \\

  \Phi_2 \\
\end{array} \right)=\mu^{-1}\Lambda^{-\frac{3}{4}}\left(\begin{array}{c}
  V_1 \\

  V_2 \\
\end{array} \right),
\]
where the linear operator $L$ is as in (\ref{eqL}), for $x\in
\mathbb{R}$ which is the natural domain of definition for $\Phi_1$
and $\Phi_2$.

In view of the righthand side of the above equation and the asymptotic behaviour
    (\ref{eqV1V2asympt}) (keep in mind (\ref{eqV1V2sym})), we are naturally led to seek $(\Phi_1,\Phi_2)$ as
\[
(\Phi_1,\Phi_2)=\mu^{-1}\Lambda^{-\frac{3}{4}}
( Z_1, Z_2)+(\tilde{\Phi}_1,\tilde{\Phi}_2),
\]
where $Z_1,Z_2$ are some smooth, fixed functions that satisfy
\begin{equation}\label{eqZi}
Z_1(x)=-\psi_0 \frac{x^3}{6}-\kappa \frac{x^2}{2},\ Z_2(x)=0, \
x\geq 1,\ \textrm{and\ (say)}\  Z_1(-x)\equiv Z_2(x).
\end{equation}
Then, the fluctuation $(\tilde{\Phi}_1,\tilde{\Phi}_2)$ should
satisfy
\begin{equation}\label{eqlinearInhomTilde}
L\left(\begin{array}{c}
  \tilde{\Phi}_1 \\

  \tilde{\Phi}_2 \\
\end{array} \right)=\mu^{-1}\Lambda^{-\frac{3}{4}}\left(\begin{array}{c}
  F_1 \\

  F_2 \\
\end{array} \right)
\end{equation}
for some fixed, smooth pair $(F_1,F_2)$ such that \begin{equation}\label{eqsectildf}
\left|F_1(x)\right|+\left|F_2(x)\right|\leq C e^{-cx^2}, \ x \in
\mathbb{R}.
\end{equation}

In the sequel, we will show how the  injectivity result in Proposition \ref{proBerestNondegen} can be used to establish the existence of solutions with \emph{linear growth} to the inhomogeneous linear problem:
\begin{equation}\label{eqLInhomo} L\left(\begin{array}{c}
			\Phi_1 \\
			
			\Phi_2 \\
		\end{array}
		\right)=\left(\begin{array}{c}
			H_1 \\
			
			H_2 \\
		\end{array}
		\right),\ \ x\in \mathbb{R},
	\end{equation}
with $H_1,H_2$ being smooth and decaying exponentially fast.
We remark that it is not clear to us how to use tools from functional analysis to achieve this because, as we expect from the general theory in \cite{hislop}, the continuous spectrum of $L$ (when defined in the natural Hilbert space) should be the  interval $[0,\infty)$ (it was also shown in \cite{berestycki-wei2012} that the whole spectrum of $L$ is nonnegative). Another obstruction is that, even though we are aware of two elements in the (formal) kernel of $L$ (recall (\ref{eqkernel})),
the remaining two elements or their asymptotic behaviour are not known to us (in fact, we suspect that the latter should involve a super-exponential growth which is not useful for matching purposes). Therefore, in contrast to related scalar second order problems (see for example \cite[Lem. 4.1]{grossi2008radial}), it is not clear how to derive conclusions from the corresponding variations of constants formula. Lastly, in relation to the cooperative character of $L$ (see \cite{berestycki2}), let us mention that we have not been able to construct appropriate upper and lower solution pairs to (\ref{eqLInhomo}).

Motivated by the existence proof of \cite{berestycki-wei2012} for the nonlinear problem (\ref{eqBUsystem}), we will first solve (\ref{eqLInhomo}) in large bounded intervals and then obtain the sought solution via a limiting procedure. In this direction, we have the following result.
Let us point out that the following estimates for the problems in the bounded intervals do not only serve as stepping stones to reach this goal but  will  also play a crucial role in the upcoming singular perturbation analysis of (\ref{eqEqGen})-(\ref{eqBdryGen}).

\begin{pro}\label{proNonSymPro}
		Given $\alpha>0$, there exist $M_0,C>0$ such that the boundary value problem
		\begin{equation}\label{eqEqToa}
		L\left( \begin{array}{c}
		\Phi_1 \\
		\Phi_2
		\end{array}
		\right)=\left( \begin{array}{c}
		H_1 \\
		H_2
		\end{array}
		\right),\ \ |x|<M;\ \ \Phi_i\left(\pm M \right)=0,\ i=1,2,
		\end{equation}
		where $L$ is as in (\ref{eqL}) and $H_1,H_2\in$ $C\left[-M,M \right]$, has a unique solution such that
		\begin{equation}\label{eqNonSymApr}
		\sum_{i=1}^{2}\left(\|\Phi'_i\|_{L^\infty\left(-M,M \right)}+\|\Phi_i\|_{L^\infty\left(-M,M \right)}\right)\leq C M\sum_{i=1}^{2}\|e^{\alpha |x|}H_i\|_{L^\infty\left(-M,M \right)},
		\end{equation}
		provided that $M \geq M_0$.

If we further assume one of the following orthogonality conditions:
\begin{equation}\label{eqOrthAssu}
\int_{-M}^{M}\left(V'_1 H_1+V'_2 H_2 \right)dx=0\ \ \textrm{or}\ \
\int_{-M}^{M}\left((xV'_1+V_1) H_1+(xV'_2+V_2) H_2 \right)dx=0,
\end{equation}
we get that
\begin{equation}\label{eq216+}
		\sum_{i=1}^{2}\left(\|\Phi'_i\|_{L^\infty\left(-M,M \right)}+\|\Phi_i\|_{L^\infty\left(-M,M \right)}\right)\leq C \sum_{i=1}^{2}\|e^{\alpha |x|}H_i\|_{L^\infty\left(-M,M \right)}.
		\end{equation}
	\end{pro}
	\begin{proof}
Since the linear operator $L$ is self-adjoint (in the natural
Sobolev spaces associated to the boundary value problem), we only need to verify the validity of the asserted a-priori estimates. An important role  will be played by the 'blow-down' problem:
		\begin{equation}\label{eqNonSymBD}
		\left\{
		\begin{array}{l}
		-\frac{d^2\varphi_1}{dy^2}+ M^2 V_2^2\left(My \right)\varphi_1+ 2M^2 V_1V_2 \left(My \right)\varphi_2= M^2 h_1, \\
		\\
		-\frac{d^2\varphi_2}{dy^2}+ M^2 V_1^2\left(My \right)\varphi_2+ 2M^2 V_1V_2 \left(My \right)\varphi_1= M^2 h_2,\\
		\\
		\textrm{for}\  |y|<1;\ \ \varphi_i(\pm 1)=0,\ \ i=1,2,
		\end{array}
		\right.
		\end{equation}
		where \[\varphi_i(y)=\Phi_i\left(My \right)  \ \ \textrm{and}  \ \ h_i(y)=H_i\left(My \right),
		\ \ i=1,2.\]

Let us start by establishing the a-priori estimate
\begin{equation}\label{eqapToa}
		\sum_{i=1}^{2}\|\Phi_i\|_{L^\infty\left(-M,M \right)}\leq C M\sum_{i=1}^{2}\|e^{\alpha |x|}H_i\|_{L^\infty\left(-M,M \right)}.
		\end{equation}
Suppose, to the contrary, that the above a-priori estimate  were false. Then, there would exist $M_n \to \infty$ and pairs
		$(\varphi_{1,n},\varphi_{2,n})\in C^2[-1,1]\times C^2[-1,1]$, $(h_{1,n},h_{2,n})\in C[-1,1]\times C[-1,1]$, satisfying (\ref{eqNonSymBD}) with $M=M_n$, which violate it. In fact, there is no loss of generality in assuming that $\|\varphi_{1,n}\|_{L^\infty(-1,1)}\geq \|\varphi_{2,n}\|_{L^\infty(-1,1)}$. Dividing both equations by $\|\varphi_{1,n}\|_{L^\infty(-1,1)}$, we may further assume that
		\begin{equation}\label{eqNonSymContra}
		\begin{array}{l}
		\|\varphi_{1,n}\|_{L^\infty(-1,1)}=1, \ \ \|\varphi_{2,n}\|_{L^\infty(-1,1)}\leq 1 \\
		\\
		\textrm{and}\ \ M_n\sum_{i=1}^{2}\|e^{\alpha \left|M_ny\right|}h_{i,n}\|_{L^\infty(-1,1)}\to 0.
		\end{array}
		\end{equation}
		Throughout the rest of the proof, $c \setminus C$ will stand for small$\setminus$large positive generic constants that are independent of $n$.
		A standard barrier argument yields that
		\begin{equation}\label{eqexpbarb}
		\left|\varphi_{1,n}(y) \right|\leq e^{cM_ny},\ \ -1\leq y \leq 0;\ \ \left|\varphi_{2,n}(y) \right|\leq e^{-cM_ny},\ \ 0\leq y \leq 1.
		\end{equation}
		In view of (\ref{eqNonSymBD}), (\ref{eqNonSymContra}) and the above relation, by a standard diagonal-compactness argument,
		passing to a subsequence if necessary, we find that
		\begin{equation}\label{eqNonSymLoc}
		\varphi_{i,n}\to \varphi_{i,\infty}\ \ \textrm{in}\ \ C_{loc}^2\left([-1,1]\setminus \{0 \}\right),\ \ i=1,2,
		\end{equation}
		where the limiting functions satisfy
		\[
		\varphi_{1,\infty}(y)=0,\ y\in [-1,0);\ \ \frac{d^2\varphi_{1,\infty}}{dy^2}=0,\ y\in (0,1],
		\]
		and
		\[
		\frac{d^2\varphi_{2,\infty}}{dy^2}=0,\ y\in [-1,0);\ \ \varphi_{2,\infty}(y)=0,\ y\in (0,1].
		\]
		Hence, we get that
		\begin{equation}\label{eqNonSymInf}
		\varphi_{1,\infty}(y)=a_1(y-1),\ \ y\in (0,1],\ \ \textrm{and}\ \ \varphi_{2,\infty}(y)=a_2(y+1),\ \ y\in [-1,0).
		\end{equation}

We will next show that the convergence in (\ref{eqNonSymLoc}) can be strengthened to
\begin{equation}\label{eqpasen}
\left|\varphi_{i,n}(y)- \varphi_{i,\infty}(y) \right|\leq Ce^{-cM_n |y|}+o(1),
\end{equation}
uniformly for $(-1)^{i+1}y \in (0,1]$, as $n\to \infty$, $i=1,2$ (keep in mind (\ref{eqexpbarb}) for the remaining intervals).
To this end, let us consider the difference
\[
\psi_{i,n}=\varphi_{i,n}- \varphi_{i,\infty}\ \ \textrm{for}\ \ y \in (0,1].
\]
Then, in view of (\ref{eqV1V2asympt}), (\ref{eqNonSymBD}) and (\ref{eqNonSymContra}),
we find that
\begin{equation}\label{eqcomag}
\left|\frac{d^2\psi_{i,n}}{dy^2} \right|=\left|\frac{d^2\varphi_{i,n}}{dy^2} \right|\leq CM_n^2e^{-cM_ny},
\ \ y \in (0,1].
\end{equation}
Note that we only made mild use of the last assumption in (\ref{eqNonSymContra}) at this point
(in particular, there was no need here for the term $M_n$ in front of the sum).
In turn, integrating twice the above relation, and making use of (\ref{eqNonSymLoc})
only at  $y=1$, yields estimate (\ref{eqpasen}) for $i=1$.
The case $i=2$ is completely analogous. Observe that (\ref{eqpasen})
provides useful information only for $|y|\gg M_n^{-1}$.

Actually, there are certain 'reflection laws' that have to be satisfied by $\varphi_{1,\infty}$ and $\varphi_{2,\infty}$  at $y=0$.
Indeed, by testing (\ref{eqNonSymBD}) with $\left(V'_1\left(M_ny \right),V'_2\left(M_ny \right) \right)$
and integrating by parts, we arrive at
\begin{equation}\label{eqIntOrthg}
\begin{split}
\sum_{i=1}^{2}\left[\frac{d\varphi_{i,n}(1)}{dy}V'_i\left(M_n \right)-\frac{d\varphi_{i,n}(-1)}{dy}V'_i\left(-M_n \right)\right]=\ \ \ \ \ \   \ \ \ \ \ \ \ \ \ \ \ \ \ \ \ \ \ \ \ \ \ \ \ \ \ \ & \\
\ \ \ \ \ \ \ \ \ \ \ \ \ \ \  \ \ \ \ \ \ \ \ \ \ \ \ \ \ \ \ \ \ \ \ \ \ \ \ \ \ -M_n^2\sum_{i=1}^{2}\int_{-1}^{1}V'_i\left(M_ny \right)h_{i,n}(y)dy.
\end{split}
\end{equation}
Letting $n\to \infty$ in the above relation, and using (\ref{eqNonSymContra}), (\ref{eqNonSymLoc}), (\ref{eqNonSymInf}), we deduce that
\begin{equation}\label{eqa1a2S}
a_1+a_2=0.
\end{equation}
Likewise, testing by $\left(M_nyV'_1\left(M_ny \right)+{V}_1\left(M_ny \right),M_nyV'_2\left(M_ny \right)+{V}_2\left(M_ny \right) \right)$ yields
\begin{equation}\label{eqa1a2D}
a_1=a_2.
\end{equation}
Hence, we get that
\begin{equation}\label{eqigno}
a_1=a_2=0.
\end{equation}
A-posteriori, it turns out that only one of the relations (\ref{eqa1a2S}), (\ref{eqa1a2D}) will be needed to reach our eventual goal (see (\ref{eqNonSymDes}) below). So, with the next assertion of the proposition in our mind, let us ignore (\ref{eqigno}) and, say, (\ref{eqa1a2D}).

On the other side, again
by the standard  diagonal-compactness
argument, passing to a further subsequence if necessary, we find that
\[
\Phi_{i,n}\to \Phi_{i,\infty}\ \ \textrm{in}\ \ C^2_{loc}(\mathbb{R}) \
\ \textrm{as}\ n\to \infty,\ \ i=1,2,
\]
where $\Phi_{1,\infty},\Phi_{2,\infty}$ satisfy
\begin{equation}\label{eqcomag22}
L\left(\begin{array}{c}
  \Phi_{1,\infty} \\
  \Phi_{2,\infty} \\
\end{array} \right)=\left(\begin{array}{c}
  0 \\
  0 \\
\end{array}\right),\ \ x\in \mathbb{R},\
\ \textrm{and}\ \
\sum_{i=1}^{2}\|\Phi_{i,\infty}\|_{L^\infty(\mathbb{R})}\leq 2.\end{equation}	Note that again we did not use the full strength of the last assumption in (\ref{eqNonSymContra}) (this argument goes through without the factor $M_n$ in front of the sum).	Thus, by virtue of Proposition \ref{proBerestNondegen},  we infer that
		\begin{equation}\label{eqbreg}
		\varphi_{i,n}\left((\ln \Lambda_n)^{-1}x \right)\to b V'_i(x)\ \ \textrm{in}\ \ C_{loc}^1(\mathbb{R}),\ \ i=1,2,
		\end{equation}
		for some $b\in \mathbb{R}$.
		
		In light of (\ref{eqexpbarb}), (\ref{eqpasen}) and (\ref{eqbreg}), to reach a contradiction, it is enough to show that
		\begin{equation}\label{eqNonSymDes}
		a_1=a_2=b=0.
		\end{equation}
		For this purpose, we will exploit that (\ref{eqpasen}) and (\ref{eqbreg}) should match where their domains of effectiveness overlap. More precisely, we will focus our attention to the points $\pm K M_n^{-1}$ in the latter intermediate zone, where $K>0$ is any sufficiently large positive number (independent of $n$). On the one hand, relations (\ref{eqNonSymInf})
		and (\ref{eqpasen}) give us that
\[
\varphi_{1,n}(K M_n^{-1})=-a_1+\mathcal{O}\left(e^{-cK}\right)+o(1)\ \ \textrm{as}\ \ n\to \infty.
\]		
On the other hand, we obtain from (\ref{eqbreg}) that		
		\[
		\varphi_{1,n}(K M_n^{-1})=bV'_1(K)+o(1)\ \ \textrm{as}\ \ n\to \infty.
		\]		
	Equating the righthand sides of the above two relations, then letting $n\to \infty$ and subsequently $K\to \infty$
	in the resulting identity, yields that
	\begin{equation}\label{eqa1b}
	-a_1=\psi_0 b.
	\end{equation}	
	In the same manner we can see that
	\begin{equation}\label{eqa2b}
		a_2=-\psi_0 b.
	\end{equation}

		The desired relation (\ref{eqNonSymDes}) now follows at once from (\ref{eqa1a2S}), (\ref{eqa1b}) and (\ref{eqa2b}).
The proof of the a-priori estimate (\ref{eqapToa}) is therefore complete.

We are now in position to establish the validity of the full a-priori estimate (\ref{eqEqToa}). In view of (\ref{eqEqToa}), (\ref{eqapToa}), (\ref{eqexpbarb}), and the asymptotic behaviour of $V_1$, $V_2$,  we find that
\[
\sum_{i=1}^{2}\|\Phi ''_i\|_{L^\infty(-M,M)}\leq CM\sum_{i=1}^{2}\|e^{\alpha |x|}H_i\|_{L^\infty(-M,M)}.
\]
The desired estimate  follows at once by plainly interpolating between (\ref{eqapToa})
and the above estimate, for example using the elementary
 inequality
\begin{equation}\label{eqInterpoNew}
\|\Phi'\|_{L^\infty(-M,M)}\leq 2
\|\Phi\|_{L^\infty(-M,M)}+\|\Phi ''\|_{L^\infty(-M,M)},
\end{equation}
(which holds for $M\geq 1$).

The proof of the second assertion is completely analogous (recall the comments below (\ref{eqcomag})
and (\ref{eqcomag22})). The only essential difference is in (\ref{eqIntOrthg}) or the corresponding
relation that gives (\ref{eqa1a2D}), where now the respective orthogonality condition in (\ref{eqOrthAssu})
implies that the righthand side is zero.
	\end{proof}

\begin{rem}
	An examination of the above proof reveals that the righthand side of (\ref{eqNonSymApr}) may also be replaced by
	\[
	C  \sum_{i=1}^{2}\left\{\|e^{\alpha |x|}H_i\|_{L^\infty\left(-M,M \right)}+M \|H_i\|_{L^1\left(-M,M \right)}\right\}.
	\]
\end{rem}

\begin{rem}
	\label{remtestz}
	It is worth noting that in the proof of the above proposition we could have also tested (\ref{eqNonSymBD}) plainly by $(y,y)$ and, using (\ref{eqexpbarb}), (\ref{eqNonSymLoc}), (\ref{eqNonSymInf}), (\ref{eqbreg}) together with Lebesgue's dominated convergence theorem, arrive at the relation
	\[a_1+a_2+\psi_0 b=0.\]
\end{rem}

\begin{rem}
In contrast, the operator $L$ with Neumann boundary conditions at $\pm M$ becomes nearly non-invertible as $M \to \infty$ because $(V'_1,V'_2)$ satisfies these conditions up to an $\mathcal{O}(e^{-M})$ small error. Nevertheless, the a-priori estimate (\ref{eq216+}) still holds, provided that we restrict ourselves within the mirror symmetric class (\ref{eqV1V2sym}).
\end{rem}

The following simple result will prove extremely useful in the sequel.

\begin{lem}\label{lemLinear}
Suppose that $u,q$ are smooth and satisfy
\[-{u''}+q(x)u=\mathcal{O}(e^{-c_0x})\ \ \textrm{as}\ \ x\to +\infty, \]
for some constant $c_0>0$. Then, the following properties hold.
\begin{itemize}
  \item $\liminf_{x\to +\infty}q(x)=+\infty$ and $u$ has at most algebraic growth \\ $\implies$ $u=\mathcal{O}(e^{-c_0x})$ as $x\to +\infty$,
  \item $q=\mathcal{O}(e^{-c_0x})$ as $x\to +\infty$\\ $\implies$ $u=a_1+b_1x+\mathcal{O}(e^{-c_1x})$ as $x\to +\infty$ for some $a_1,b_1\in \mathbb{R}$ and any $c_1\in (0,c_0)$.
\end{itemize}
\end{lem}
\begin{proof}
The first property can be shown as in \cite[Lem. 7.3]{del2010toda}, while the second as in \cite[Lem. 3.2]{agudelo}.
\end{proof}
An important consequence of Proposition \ref{proNonSymPro} and Lemma \ref{lemLinear} is the next lemma, which may be
considered as a sort of \emph{Fredholm alternative} for $L$.
\begin{lem}\label{lemFred}
Assume that the components of $(H_1,H_2)\in  \left[C(\mathbb{R})\right]^2$ satisfy
\begin{equation}\label{eqExp}
|H_i(x)|\leq Ce^{-c|x|},\ \ x\in \mathbb{R},\ i=1,2,
\end{equation}
for some constants $c,C>0$, and one of the orthogonality conditions
\begin{equation}\label{eqOrthogV}
\int_{-\infty}^{\infty}\left(V'_1H_1+V'_2H_2 \right)dx=0\ \ \textrm{or}\ \ \int_{-\infty}^{\infty}\left((xV'_1+{V}_1)H_1+(xV'_2+{V}_2)H_2 \right)dx=0.
\end{equation}
Then, there exists a  solution $(\Phi_1,\Phi_2)\in \left[C^2(\mathbb{R})\right]^2$ to (\ref{eqLInhomo}) such that
\begin{equation}\label{eqRuns}
\begin{array}{lll}
  \Phi_1(x)=a_++\mathcal{O}(e^{-c'x}),&\Phi_2(x)=\mathcal{O}(e^{-c'x})& \textrm{as}\ \ x\to +\infty; \\
    & & \\
  \Phi_1(x)=\mathcal{O}(e^{c'x}),& \Phi_2(x)=a_-+\mathcal{O}(e^{c'x})& \textrm{as}\ \ x\to -\infty.
\end{array}
\end{equation}
for some $a_\pm \in \mathbb{R}$ such that
\begin{equation}\label{eqa+a-}
a_++a_-=-\frac{1}{2\psi_0}\int_{-\infty}^{\infty}\left((xV'_1+{V}_1)H_1+(xV'_2+{V}_2)H_2 \right)dx
\end{equation}
and for any $c' \in (0,c)$.
\end{lem}
\begin{proof}
Let us begin by assuming that the first orthogonality condition in (\ref{eqOrthogV}) holds.
We will construct the desired solution through a limiting process. Motivated by the second assertion of Proposition \ref{proNonSymPro}, we consider the following sequence of approximate problems:
\begin{equation}\label{eqLInhomoNN} L\left(\begin{array}{c}
		\Phi_1 \\
		
		\Phi_2 \\
	\end{array}
	\right)=\left(\begin{array}{c}
		H_{1,n} \\
		
		H_{2,n} \\
	\end{array}
	\right),\ \ x\in (-n,n);\ \ \Phi_i(\pm n)=0,\ i=1,2,\ n\geq 1,
\end{equation}
where
\[
\left(\begin{array}{c}
H_{1,n} \\

H_{2,n} \\
\end{array}
\right)=\left(\begin{array}{c}
H_{1} \\

H_{2} \\
\end{array}
\right)-d_n \left(\begin{array}{c}
V'_1 \\

V'_2 \\
\end{array}
\right)e^{-c|x|},
\]
where
\[
d_n=\frac{\int_{-n}^{n}\left(V'_1H_1+V'_2H_2 \right)dx}{\int_{-n}^{n}\left[(V'_1)^2+(V'_2)^2\right] e^{-c|x|}dx}
\]
is chosen so that
\begin{equation}\label{eqOrthogVnnn}
	\int_{-n}^{n}\left(H_{1,n}V'_1+H_{2,n}V'_2 \right)dx=0.
\end{equation}
We note that (\ref{eqExp}), (\ref{eqOrthogV}) and Lebesgue's dominated convergence theorem yield that
\[
d_n \to 0.
\]

By the second assertion of Proposition \ref{proNonSymPro}, if $n$ is sufficiently large, there exists a solution
$(\Phi_{1,n},\Phi_{2,n})$  to (\ref{eqLInhomoNN}) such that
\[
\sum_{i=1}^{2}\left\{\|\Phi'_{i,n} \|_{L^\infty(-n,n)}+\|\Phi_{i,n} \|_{L^\infty(-n,n)}\right\}\leq C,
\]
for some constant $C>0$ that is independent of $n$. Hence, thanks again to the standard diagonal-compactness argument, letting $n\to \infty$ in (\ref{eqLInhomoNN}) (along the appropriate subsequence) yields a bounded solution to (\ref{eqLInhomo}). The asymptotic behaviour (\ref{eqRuns}) is a direct consequence of Lemma \ref{lemLinear}. Lastly, relation (\ref{eqa+a-}) follows at once by testing (\ref{eqLInhomo}) with $(xV'_1+V_1,xV'_2+V_2)$.

The proof in the case of the second orthogonality condition in (\ref{eqOrthogV}) is completely analogous.
\end{proof}
We can now establish our main result concerning the solvability properties of (\ref{eqLInhomo}).
\begin{pro}\label{proSymA}
	Given $(H_1,H_2)\in  \left[C(\mathbb{R})\right]^2$ satisfying the exponential decay estimate
	(\ref{eqExp}), there exists a  solution $(\Phi_1,\Phi_2)\in \left[C^2(\mathbb{R})\right]^2$ to (\ref{eqLInhomo}) such that
	\[\begin{array}{lll}
	\Phi_1(x)=a_++bx+\mathcal{O}(e^{-c'x}), & \Phi_2(x)=\mathcal{O}(e^{-c'x}) & \textrm{as}\ \ x\to +\infty;
	\\
	& & \\
	\Phi_1(x)=\mathcal{O}(e^{c'x}),& \Phi_2(x)=a_-+bx+\mathcal{O}(e^{c'x})& \textrm{as}\ \ x\to -\infty,
	\end{array}
	\]
	for  any $c' \in (0,c)$, where
	\[
	b=-\frac{1}{2\psi_0}\int_{-\infty}^{\infty}\left(V'_1H_1+V'_2H_2 \right)dx,
	\]
	and $a_+,a_-$ satisfy (\ref{eqa+a-}).
\end{pro}
\begin{proof}
The main idea is to search for a solution in the form
\[
\left(\begin{array}{c}
\Phi_1 \\

\Phi_2 \\
\end{array}
\right)=B\left(\begin{array}{c}
V_1 \\

-V_2 \\
\end{array}
\right)+\left(\begin{array}{c}
\Psi_1 \\

\Psi_2 \\
\end{array}
\right)
\]
with $B\in \mathbb{R}$ and $(\Psi_1,\Psi_2)\in \left[C^2(\mathbb{R})\right]^2$.
The new equation  that now needs to be satisfied is
\[
L\left(\begin{array}{c}
\Psi_1 \\

\Psi_2 \\
\end{array}
\right)=\left(\begin{array}{c}
H_1 \\

H_2 \\
\end{array}
\right)+B\left(\begin{array}{c}
2V_1 V_2^2 \\

-2V_2 V_1^2 \\
\end{array}
\right).
\]
To conclude, we can apply Lemma \ref{lemFred}  after making the choice
\[
B=-\frac{\int_{-\infty}^{\infty}\left(V'_1H_1+V'_2H_2 \right)dx}{2\int_{-\infty}^{\infty}\left(V_1 V_2^2 V'_1-V_2 V_1^2 V'_2\right)dx}\stackrel{(\ref{eqBUsystem}),(\ref{eqBUasymp})}{=}-\frac{1}{2\psi_0^2}\int_{-\infty}^{\infty}
\left(V'_1H_1+V'_2H_2 \right)dx.
\]
\end{proof}

By applying Proposition \ref{proSymA}  to the case where the
righthand side of (\ref{eqLInhomo}) is the pair $(F_1,F_2)$ (which actually is independent of $\Lambda$ and satisfies (\ref{eqsectildf})), as defined through
(\ref{eqlinearInhomTilde}),  we obtain the existence
of a solution $(\hat{\Phi}_1,\hat{\Phi}_2)$ to  the system
\begin{equation}\label{eqPhiHatDef}
L\left(\begin{array}{c}
  \hat{\Phi}_1 \\

   \hat{\Phi}_2 \\
\end{array} \right)=\left(\begin{array}{c}
  F_1 \\

  F_2 \\
\end{array} \right),\ \ x\in \mathbb{R},
\end{equation}
such that
\begin{equation}\label{eqPhiHatAsympt}\begin{array}{lll}
                                         \hat{\Phi}_1(x)=a_++bx+\mathcal{O}(e^{-Dx}),&  \hat{\Phi}_2(x)=\mathcal{O}(e^{-Dx})&  \textrm{as}\ x\to
+\infty; \\
                                          &   \\
 \hat{\Phi}_1(x)=\mathcal{O}(e^{Dx}),        & \hat{\Phi}_2(x)=a_-+bx+\mathcal{O}(e^{Dx})&  \textrm{as}\ x\to
-\infty,
                                      \end{array}
\end{equation}
for some $a_\pm,b\in \mathbb{R}$ and any $D>0$ (the expressions for the sum $a_++a_-$ and $b$ which are provided by the aforementioned proposition, with $(H_1,H_2)$ in place of $(F_1,F_2)$, will not be needed).
In fact, the above relations can be differentiated arbitrary many times.

This allows us to improve our inner approximate solution (\ref{eqvinprelim}):
\begin{definition}\label{definner}
We define the
inner approximate  solution to (\ref{eqEqGen}) as $(v_{1,in},v_{2,in})$, with
\begin{equation}\label{eqvinDef}
v_{i,in}(z)=\mu
\Lambda^{-\frac{1}{4}}V_i(x)+\mu^{-1}\Lambda^{-\frac{3}{4}}\left[Z_i(x)+\hat{\Phi}_i(x)
\right]+BE_i(x),\ \ |z|\leq (\ln \Lambda)\Lambda^{-\frac{1}{4}},\
i=1,2,
\end{equation} where $x$ is the stretched variable (\ref{eqx}), $V_i,Z_i,\hat{\Phi}_i$ are defined through  Proposition
\ref{proBerestycki}, (\ref{eqZi}), (\ref{eqPhiHatDef})-(\ref{eqPhiHatAsympt})
respectively and
\begin{equation}\label{eqEi}
(E_1,E_2)=(xV'_1+V_1,xV'_2+V_2)
\end{equation}
is the second element of the kernel of $L$ from (\ref{eqkernel}). The constants  $\mu$ and $B$ will be
determined later, subject to the constraints (\ref{eqmuprelim})
and
\begin{equation}\label{eqBiprelim}
B=\mathcal{O}(\Lambda^{-\frac{3}{4}})\ \ \textrm{as}\ \ \Lambda\to
\infty,
\end{equation}
respectively.\end{definition}

\begin{rem}\label{remMuB}
It may appear at first sight that the above inner approximate solution contains two free parameters,
$\mu$ and $B$. However, keep in mind that both are present due to the same reason, namely the invariance of (\ref{eqBUsystem}) under scaling. So, essentially there is only one free parameter. An analogous remark applies to the outer approximate solution in (\ref{eqOutDef}). On the other hand, we stress that, in principle, all the aforementioned parameters  should be present when carrying out the formal matched asymptotic analysis.
\end{rem}

\subsubsection{The remainder of the inner approximate  solution}\label{subsubsecRemIn}
In view of (\ref{eqx}), (\ref{eqmuprelim}), (\ref{eqRemInBasic}),
(\ref{eqBiprelim}), and the construction of  $(\hat{\Phi}_1,\hat{\Phi}_2)$, we
have the validity of the following lemma.
\begin{lem}\label{lemRemIn} In equation (\ref{eqEqGen}),
 the remainder
\[
R(v_{1,in},v_{2,in})=\left(\begin{array}{c}
  -v_{1,in}''+v_{1,in}^3-v_{1,in}+\Lambda v_{2,in}^2 v_{1,in}  \\
    \\
  -v_{2,in}''+v_{2,in}^3-v_{2,in}+\Lambda v_{1,in}^2 v_{2,in}  \\
\end{array}
 \right)
\]
which is left by the solution $(v_{1,in},v_{2,in})$ of Definition \ref{definner} satisfies
\[
R(v_{1,in},v_{2,in})=\mathcal{O}(\Lambda^{-\frac{3}{4}})\left(\begin{array}{c}
   \Lambda^{\frac{3}{4}}z^3+1 \\
    \\
 e^{-D\Lambda^\frac{1}{4}z} \\
\end{array}
 \right),
\]
for any $D\geq 1$, uniformly on $\left[0,(\ln
\Lambda)\Lambda^{-\frac{1}{4}} \right]$, as $\Lambda \to \infty$.
An analogous estimate holds on $\left[-(\ln
\Lambda)\Lambda^{-\frac{1}{4}},0 \right]$.
\end{lem}
\section{Matching the outer and inner approximate solutions}\label{secMatch}
In this section we will 'stitch' together   the outer and inner approximate solutions
 by suitably adjusting the parameters $\xi_1,\xi_2,\tau_1,\tau_2,\mu, B$
in their definitions (recall (\ref{eqOutDef}) and (\ref{eqvinDef})), subject to the constraints
(\ref{eqconstr}), (\ref{eqmuprelim}) and (\ref{eqBiprelim}). Classical singular perturbation theory dictates that this must be done so
that the inner and outer approximations are sufficiently close in the $C^1$-sense over some intermediate zone satisfying
$\Lambda^{-\frac{1}{4}}\ll |z|\ll 1$ (see the section on matched asymptotic expansions in any
textbook on the subject or the so called exchange lemmas of the modern geometric singular perturbation theory). This property is of course already satisfied by the first components of the aforementioned approximations in the negative part of the intermediate zone, and the analogous property holds in the positive part.
Thus, the task of matching the inner and outer approximate solutions in $C^1$ over an intermediate zone amounts to satisfying a total of four algebraic equations (one for each of the first two terms of the Taylor expansions of the non-trivial outer approximations). However, in view of Remark \ref{remMuB}, we essentially have only three free parameters to adjust for this purpose. Fortunately, with some care, this overdetermined issue can be resolved by exploiting the conservation of the hamiltonian of (\ref{eqEqGen}).  Actually,
%for reasons that will become apparent (ultimately related to the first case of Lemma \ref{lemUnondeg}),
it is more convenient to  match them continuously as best as possible at just the two boundary points
$\pm(\ln \Lambda)\Lambda^{-\frac{1}{4}}$. It turns out that the  algebraic system which arises from these considerations,  comprising of three equations   (one at each boundary point together with  one  from the explotation of the hamiltonian structure, say at the origin) and containing three unknowns (essentially coming from the translation and scaling invariances of the outer and inner limit problems respectively),  is solvable for large $\Lambda$ thanks to the fact that $\psi_0\neq 0$. In fact, this type of matching leads us naturally to building a  solution of (\ref{eqEqGen})-(\ref{eqBdryGen}) in the same spirit, that is by constructing separately  inner and outer genuine solutions which match continuously at $\pm(\ln \Lambda)\Lambda^{-\frac{1}{4}}$ and share the same hamiltonian constant. In particular, the latter strategy allows us to use directly the last observation in Lemma \ref{lemUnondeg} and  Proposition \ref{proBerestNondegen} for this purpose.

\subsection{Matching $(v_{1,out},v_{2,out})$ and
$(v_{1,in},v_{2,in})$ continuously at $\pm (\ln
\Lambda)\Lambda^{-\frac{1}{4}}$}\label{subsecMatchCont} In view of
(\ref{eqGenPsi0}), (\ref{eqOutDef}), (\ref{eqconstr}) and the facts
that \begin{equation}\label{eqUfacts}U_1''(0)=0,\ \
 U_1'''(0)=-\psi_0,\ \ U_1^{(4)}(0)=0,\end{equation} we find that
\[
\begin{array}{rcl}
  v_{1,out}\left((\ln \Lambda)\Lambda^{-\frac{1}{4}} \right) & = & \psi_0 \left((\ln \Lambda)\Lambda^{-\frac{1}{4}}+\xi_1
\right)-\frac{\psi_0}{6}\left((\ln
\Lambda)\Lambda^{-\frac{1}{4}}+\xi_1 \right)^3 \\
    &   &   \\
    &   & +\mathcal{O}\left((\ln \Lambda)^5\Lambda^{-\frac{5}{4}}
\right)+\tau_1 \psi_0+\tau_1 \mathcal{O}\left((\ln
\Lambda)^2\Lambda^{-\frac{1}{2}} \right)
 \\
\end{array}
\]
as $\Lambda \to \infty$, where the quantities in the Landau
symbols are independent of $\tau_1$. In turn, by expanding, we get
that
\[
\begin{array}{rcl}
  v_{1,out}\left((\ln \Lambda)\Lambda^{-\frac{1}{4}} \right) & = & \psi_0 \xi_1+\psi_0(\ln \Lambda)\Lambda^{-\frac{1}{4}}+\tau_1
  \psi_0-\frac{\psi_0}{6}\xi_1^3-\frac{\psi_0}{2}(\ln
\Lambda)\Lambda^{-\frac{1}{4}}\xi_1^2
\\
& &\\
& &-\frac{\psi_0}{6}(\ln
\Lambda)^3\Lambda^{-\frac{3}{4}}-\frac{\psi_0}{2}(\ln
\Lambda)^2\Lambda^{-\frac{1}{2}}\xi_1 \\
    &   &   \\
    &   &+\mathcal{O}\left((\ln
\Lambda)^5\Lambda^{-\frac{5}{4}} \right)+\tau_1
\mathcal{O}\left((\ln \Lambda)^2\Lambda^{-\frac{1}{2}} \right)
 \\
\end{array}
\]
as $\Lambda \to \infty$.

 On the other side, from
(\ref{eqV1V2asympt}), (\ref{eqZi}), (\ref{eqPhiHatAsympt}) and
(\ref{eqvinDef}), we obtain that
\[\begin{array}{rcl}
  v_{1,in}\left((\ln \Lambda)\Lambda^{-\frac{1}{4}} \right) & = &  \mu^2\psi_0
(\ln \Lambda) \Lambda^{-\frac{1}{4}}+\mu \kappa
\Lambda^{-\frac{1}{4}}+B\kappa +2B\psi_0\mu (\ln \Lambda)+\mu^{-1}\Lambda^{-\frac{3}{4}}a_+ \\
    &   &   \\
   &   &+b(\ln
\Lambda)\Lambda^{-\frac{3}{4}} -\frac{\psi_0}{6}\mu^2(\ln
\Lambda)^3\Lambda^{-\frac{3}{4}}-\frac{1}{2}\kappa \mu (\ln
\Lambda)^2

\Lambda^{-\frac{3}{4}}+\mathcal{O}(\Lambda^{-\infty}) \\
\end{array}
\]
 as $\Lambda \to \infty$.

Given $B$
satisfying (\ref{eqBiprelim}), we take
\begin{equation}\label{eqmuView}
\mu=1,\ \ \xi_1=\psi_0^{-1}\kappa \Lambda^{-\frac{1}{4}}+2B (\ln
\Lambda)+\frac{1}{2}(\ln \Lambda)\Lambda^{-\frac{1}{4}}\xi_1^2,
\end{equation}
which is indeed possible by the implicit function theorem for $\Lambda$ large (we can even find an explicit formula for $\xi_1$ by solving the above trinomial). In
turn, we choose
\[
\tau_1=\frac{\xi_1^3}{6}+\psi_0^{-1}B \kappa+\psi_0^{-1}a_+
\Lambda^{-\frac{3}{4}}+\psi_0^{-1}b(\ln
\Lambda)\Lambda^{-\frac{3}{4}}.
\]
Then, using that
\[
\xi_1=\psi_0^{-1}\kappa \Lambda^{-\frac{1}{4}}+\mathcal{O}\left((\ln
\Lambda)\Lambda^{-\frac{3}{4}} \right)\ \ \textrm{as}\ \
\Lambda\to \infty,
\]
it follows readily that
\begin{equation}\label{eqv1(out-in)}
(v_{1,out}-v_{1,in})\left((\ln \Lambda)\Lambda^{-\frac{1}{4}}
\right)=\mathcal{O}\left((\ln \Lambda)^5\Lambda^{-\frac{5}{4}}
\right)\ \ \textrm{as}\ \ \Lambda\to \infty.\end{equation}
We have
\begin{equation}\label{eqv2(out-in)}
(v_{2,out}-v_{2,in})\left((\ln \Lambda)\Lambda^{-\frac{1}{4}}
\right)=-v_{2,in}\left((\ln \Lambda)\Lambda^{-\frac{1}{4}}
\right)=\mathcal{O}\left(\Lambda^{-\infty} \right)\ \ \textrm{as}\ \
\Lambda\to \infty.\end{equation}
Analogous considerations apply at $-(\ln \Lambda)\Lambda^{-\frac{1}{4}}$.
\subsection{Adjusting the value of the Hamiltonian on the inner approximate solution at
$z=0$}\label{subsecAdjHam1} In this subsection, we will choose $B$, under the constraint (\ref{eqBiprelim}),
so that the value of the Hamiltonian on $(v_{1,in},v_{2,in})$ at
$z=0$ is equal to the Hamiltonian constant of the expected
heteroclinic connection, namely $-\psi_0^2/2$.

Firstly, from (\ref{eqx}), (\ref{eqvinDef}) and
(\ref{eqBiprelim}), we note that
\[
\left[v_{i,in}'(0)\right]^2=\left[V'_i(0)\right]^2+2V'_i(0)\left[O(\Lambda^{-\frac{1}{2}})+2V'_i(0)B\Lambda^\frac{1}{4}
\right]+\mathcal{O}(\Lambda^{-1}),\ \ i=1,2,
\]
as $\Lambda \to \infty$, with $O(\Lambda^{-\frac{1}{2}})$ being
independent of $B$. Furthermore, it is clear that
\[
v_{i,in}^4(0)=\Lambda^{-1}V_i^4(0)+\mathcal{O}(\Lambda^{-\frac{1}{2}})B^2+4\Lambda^{-\frac{3}{4}}V_i^4(0)B+O(\Lambda^{-\frac{3}{2}})\
\ \textrm{as}\ \Lambda \to \infty,
\]
where $O(\Lambda^{-\frac{3}{2}})$ is independent of $B$, and that
\[
\frac{\left(1-v_{i,in}^2(0)
\right)^2}{4}=\frac{1}{4}+O(\Lambda^{-\frac{1}{2}})+\mathcal{O}(\Lambda^{-\frac{1}{4}})B\stackrel{(\ref{eqGenPsi0})}{=}
\frac{\psi_0^2}{2}+O(\Lambda^{-\frac{1}{2}})+\mathcal{O}(\Lambda^{-\frac{1}{4}})B
\]
as $\Lambda\to \infty$, $i=1,2$, where $O(\Lambda^{-\frac{1}{2}})$ is
independent of $B$. Now, using that $V_1(0)=V_2(0)$,
${V}'_1(0)=-V'_2(0)$ and the hamiltonian identity
\[
(V'_1)^2+(V'_2)^2-V_1^2V_2^2=\psi_0^2,\ \ x\in \mathbb{R},
\]
it follows readily that the sought after equality
\[
H(v_{1,in}(0),v_{2,in}(0))=-\frac{\psi_0^2}{2}\ \ (\textrm{with\
the\ obvious\ notation,\ keep\ in\ mind\ (\ref{eqHamilton}),\ (\ref{eqGenPsi0})})
\]
takes the form
\[
2\psi_0^2B
\Lambda^\frac{1}{4}=O(\Lambda^{-\frac{1}{2}})+\mathcal{O}(\Lambda^{-\frac{1}{4}})B+\mathcal{O}(\Lambda^{-1})\
\ \textrm{as}\ \Lambda\to \infty.
\]
The above equation clearly has a unique solution
\begin{equation}\label{eqBita}
B=\mathcal{O}(\Lambda^{-\frac{3}{4}})\ \ \textrm{as}\ \Lambda\to
\infty,
\end{equation}
as desired (keep in mind that, according to our notation,
the term $O(\Lambda^{-\frac{1}{2}})$ above does not contain $B$).
\subsection{A refined inner approximate solution $(w_{1,in},w_{2,in})$}\label{secRefIn}
For $\tilde{B}\in \mathbb{R}$ to be chosen later, subject to the
constraint
\begin{equation}\label{eqBtildeConstr}
\tilde{B}=\mathcal{O}(\Lambda^{-1})\ \ \textrm{as}\ \Lambda \to
\infty,
\end{equation}
we consider the more refined inner approximate solution $(w_{1,in},w_{2,in})$
\begin{equation}\label{innerwi}
w_{i,in}(z)=v_{i,in}(z)+\tilde{B}E_i(x),\ \ |z|\leq (\ln
\Lambda)\Lambda^{-\frac{1}{4}},\  i=1,2,\end{equation} where $(v_{1,in},v_{2,in})$
 is defined in Definition \ref{definner} and $E_i$ comes from (\ref{eqEi}).
 Actually, $\tilde{B}$
will turn out to be chosen much smaller than in
(\ref{eqBtildeConstr}).
%\subsection{The remainder of the refined inner approximate solution}\label{subsecRemRefIn}

It is easy to see that the assertion of Lemma \ref{lemRemIn}, concerning the remainder of this refined inner solution, continues
to hold for $(w_{1,in},w_{2,in})$.
\section{Solution of the inner problem}\label{secSolInner}
In this section, we will show that the one-parameter family of refined inner
approximate solutions $(w_{1,in},w_{2,in})$, described in the
previous section (parameterized by $\tilde{B}$), can be perturbed
smoothly to a one-parameter family of inner genuine solutions to the
system (\ref{eqEqGen}), for large $\Lambda>0$. Then, we will show
that there exists at least one value of $\tilde{B}$, in the range
(\ref{eqBtildeConstr}), for which the corresponding inner genuine solution to
(\ref{eqEqGen}) has a Hamiltonian constant equal to $-\psi_0^2/2$.
\subsection{The perturbation argument}\label{subsecPertIn}
Given $\tilde{B}$ satisfying (\ref{eqBtildeConstr}), we seek a solution of  system (\ref{eqEqGen}) as
\begin{equation}\label{eqGenForm}
({\textbf{v}_{1,in},\textbf{v}_{2,in}})=(w_{1,in},w_{2,in})+(\varphi_1,\varphi_2),
\ \ |z|\leq (\ln \Lambda)\Lambda^{-\frac{1}{4}},
\end{equation}
with
\[
\varphi_i\left(\pm (\ln \Lambda)\Lambda^{-\frac{1}{4}}\right)=0,\
\ i=1,2.
\]
%and \[\varphi_1(-z)=\varphi_2(z),\ \   |z|\leq (\ln
%\Lambda)\Lambda^{-\frac{1}{4}}.\]
After rearranging terms, we find
that $(\varphi_1,\varphi_2)$ has to satisfy
\begin{equation}\label{eqExInBas}\left\{\begin{array}{c}
\mathcal{L}(\varphi_1,\varphi_2)=-R(w_{1,in},w_{2,in})-Q(\varphi_1,\varphi_2)-N(\varphi_1,\varphi_2),
 \\
   \\
  \varphi_i\left(\pm (\ln \Lambda)\Lambda^{-\frac{1}{4}}\right)=0,\
\ i=1,2, \\
\end{array}\right.
\end{equation}
where
\[
\mathcal{L}(\varphi_1,\varphi_2)=\left(\begin{array}{c}
  -\varphi_1''+\Lambda^\frac{1}{2}V_2^2(x)\varphi_1+2\Lambda^\frac{1}{2} V_1(x)V_2(x)\varphi_2 \\
    \\
  -\varphi_2''+\Lambda^\frac{1}{2}V_1^2(x)\varphi_2+2\Lambda^\frac{1}{2} V_1(x)V_2(x)\varphi_1 \\
\end{array} \right),
\]
the term $R(w_{1,in},w_{2,in})$ denotes the remainder which is left by
$(w_{1,in},w_{2,in})$ in (\ref{eqEqGen}) (analogously to Lemma \ref{lemRemIn}),
\[Q(\varphi_1,\varphi_2)=\]\[
\left( \begin{array}{c}
  (3w_{1,in}^2-1)\varphi_1+\Lambda
\left(w_{2,in}^2-\Lambda^{-\frac{1}{2}}V_2^2
\right)\varphi_1+2\Lambda \left(w_{1,in}w_{2,in}
-\Lambda^{-\frac{1}{2}}V_1V_2  \right)\varphi_2 \\
    \\
  (3w_{2,in}^2-1)\varphi_2+\Lambda
\left(w_{1,in}^2-\Lambda^{-\frac{1}{2}}V_1^2
\right)\varphi_2+2\Lambda \left(w_{1,in}w_{2,in}
-\Lambda^{-\frac{1}{2}}V_1V_2  \right)\varphi_1 \\
\end{array}
\right)
\]
and
\[N(\varphi_1,\varphi_2)=
\left( \begin{array}{c}
  \varphi_1^3+3 w_{1,in}\varphi_1^2+\Lambda w_{1,in} \varphi_2^2+ \Lambda \varphi_2^2\varphi_1+2\Lambda w_{2,in}\varphi_1\varphi_2 \\
    \\
\varphi_2^3+3 w_{2,in}\varphi_2^2+\Lambda w_{2,in} \varphi_1^2+ \Lambda \varphi_1^2\varphi_2+2\Lambda w_{1,in}\varphi_1\varphi_2\\
\end{array}
\right).
\]

Concerning the linear operator $\mathcal{L}$, we observe that
Proposition \ref{proNonSymPro}, after a simple re-scaling (recall that $x=\Lambda^\frac{1}{4}z$), yields
the following.
\begin{cor}\label{corlinBasic}
Given $\alpha>0$, there exist $\Lambda_0,C>0$ such that the
boundary value problem
\[
\mathcal{L}\left(\begin{array}{c}
  \varphi_1 \\
  \varphi_2 \\
\end{array} \right)=\left(\begin{array}{c}
  h_1 \\
  h_2 \\
\end{array}\right),\ \ |z|<(\ln \Lambda)\Lambda^{-\frac{1}{4}};
\ \ \varphi_i\left(\pm (\ln \Lambda)\Lambda^{-\frac{1}{4}}
\right)=0,\ i=1,2,
\]
where $h_1,h_2\in C\left[-(\ln \Lambda)\Lambda^{-\frac{1}{4}},(\ln
\Lambda)\Lambda^{-\frac{1}{4}} \right]$,
has a unique solution such that
\[
\sum_{i=1}^{2}\left(\Lambda^{-\frac{1}{4}}\|\varphi_i' \|_{L^\infty(I_\Lambda)}+\|\varphi_i \|_{L^\infty(I_\Lambda)}\right)\leq C
\Lambda^{-\frac{1}{2}+\alpha} \sum_{i=1}^{2}\|h_i
\|_{L^\infty(I_\Lambda)},
\]
where
\[
I_\Lambda=\left(-(\ln \Lambda)\Lambda^{-\frac{1}{4}},(\ln
\Lambda)\Lambda^{-\frac{1}{4}} \right),
\]
provided that $\Lambda\geq \Lambda_0$.
\end{cor}

On the other side,  the remainder $R(w_{1,in},w_{2,in})$ clearly satisfies the thesis of Lemma \ref{lemRemIn}. Furthermore, using that \[
\left|v_{i,in}(z)\right|\leq C(\ln \Lambda)  \Lambda^{-\frac{1}{4}},\ \  \left|E_i(x)\right|\leq C(\ln \Lambda),
\]
together with the easy to prove estimates
\[
\left| v^2_{i,in}(z)-\Lambda^{-\frac{1}{2}} V^2_i(x) \right|+\left|v_{1,in}v_{2,in}(z)-\Lambda^{-\frac{1}{2}}V_1V_2(x) \right|\leq C (\ln \Lambda)^4  \Lambda^{-\frac{5}{4}},
\]
for $|z|\leq (\ln \Lambda)  \Lambda^{-\frac{1}{4}}$, $i=1,2$, and (\ref{eqBtildeConstr}), it follows readily that there
exists $C>0$ such that
\begin{equation}\label{eqBestim}
\sum_{i=1}^{2}\|Q_i(\varphi_1,\varphi_2) \|_{L^\infty(I_\Lambda)}\leq C \sum_{i=1}^{2}\|\varphi_i \|_{L^\infty(I_\Lambda)},
\end{equation}
 for any $\varphi_1,\varphi_2 \in
C(\overline{I_\Lambda})$.
Moreover, there
exists a $C>0$ such that
\begin{equation}\label{eqNestim}
\|N_i(\varphi_1,\varphi_2) \|_{L^\infty(I_\Lambda)}\leq C
\sum_{i=1}^{2}\left\{\|\varphi_i
\|_{L^\infty}^3+\Lambda^\frac{3}{4}(\ln \Lambda)\|\varphi_i
\|^2_{L^\infty}+\Lambda \|\varphi_i \|_{L^\infty}\|\varphi_{i+1}
\|_{L^\infty}^2 \right\},
\end{equation}
$i=1,2$,
 for any $\varphi_1,\varphi_2 \in
C(\overline{I_\Lambda})$ (with the obvious notation), and
\begin{equation}\label{eqGen44}\begin{array}{c}
 \sum_{i=1}^{2} \|N_i(\varphi_1,\varphi_2)-N_i(\psi_1,\psi_2)
\|_{L^\infty(I_\Lambda)}\leq \\
   \\
  C
\sum_{i=1}^{2}\left\{\Lambda\left(\|\varphi_i
\|_{L^\infty}^2+\|\psi_i
\|_{L^\infty}^2\right)+\Lambda^\frac{3}{4}(\ln
\Lambda)\left(\|\varphi_i \|_{L^\infty}+\|\psi_i \|_{L^\infty} \right)\right\}\\
\ \ \ \ \ \ \ \ \ \ \ \ \ \ \ \ \ \quad \ \ \ \ \ \ \ \ \ \ \ \ \ \ \ \ \ \ \ \ \ \ \ \ \ \ \ \ \ \ \ \ \ \ \  \times \left(\sum_{i=1}^{2}\|\varphi_i-\psi_i \|_{L^\infty}\right), \\
\end{array}
\end{equation}
 for any $\varphi_1,\varphi_2,\psi_1,\psi_2\in
C(\overline{I_\Lambda})$.

In view of the above, and paying attention to the dependence on $\tilde{B}$, a standard application of the contraction
mapping principle yields the following.

\begin{pro}\label{proExistInn}
Given $\alpha\in (0,1)$, there exists $C>0$ such that problem
(\ref{eqExInBas}) has a unique solution satisfying
\[
\sum_{i=1}^{2}\left(\Lambda^{-\frac{1}{4}}\|\varphi_i' \|_{L^\infty(I_\Lambda)}+\|\varphi_i \|_{L^\infty(I_\Lambda)}\right)\leq C
\Lambda^{-\frac{5}{4}+\alpha},
\]
 provided that
$\Lambda$ is sufficiently large. Moreover, this solution depends
continuously, with respect to the $C^1(\overline{I_\Lambda})$-norm, on $\tilde{B}$ as in (\ref{eqBtildeConstr}) (for fixed $\Lambda$).
\end{pro}

We point out that the aforementioned continuous dependence on $\tilde{B}$ can be proven easily as follows. Let
$\tilde{B}_n$ satisfy (\ref{eqBtildeConstr}), for fixed $\Lambda$ as in the above proposition, and $\tilde{B}_n\to \tilde{B}_\infty$
as $n\to \infty$. We denote by $\left(\varphi_{1,n},\varphi_{2,n} \right)$ and $\left(\varphi_{1,\infty},\varphi_{2,\infty} \right)$
the solutions of (\ref{eqExInBas}) corresponding to $\tilde{B}_n$ and $\tilde{B}_\infty$ respectively, as provided by the first part of the above proposition.
Then, thanks to Arzela-Ascoli's theorem, passing to a subsequence if necessary, and utilizing the uniqueness assertion of the aforementioned proposition, we find that
$\varphi_{i,n}\to \varphi_{i,\infty}$ in $C^1(\overline{I_\Lambda})$ as $n\to \infty$, $i=1,2$. Finally, by employing once more the uniqueness property of $\left(\varphi_{1,\infty},\varphi_{2,\infty} \right)$, we deduce that the previous convergence holds for the original sequence.
\subsubsection{Some preliminary positivity and monotonicity properties of the inner genuine solution
$(\textbf{v}_{1,in},\textbf{v}_{2,in})$}\label{subsubsecMonotIn}
It is  clear from the construction of the refined inner approximate solution $(w_{1,in},w_{2,in})$ and Proposition \ref{proExistInn} that, given any $L>1$, there exists $c_L>0$
such that
\begin{equation}\label{eqv2posIn}
\textbf{v}_{2,in}\geq c_L \Lambda^{-\frac{1}{4}}\ \ \textrm{and} \
\  -\textbf{v}_{2,in}'\geq c_L \ \ \textrm{on}\ \left[-(\ln
\Lambda)\Lambda^{-\frac{1}{4}},L\Lambda^{-\frac{1}{4}}\right],
\end{equation}
provided that $\Lambda>0$ is sufficiently large. On the other
side, we observe that $\textbf{v}_{2,in}$ satisfies a linear equation of the form
\begin{equation}\label{eqv2Lin}
-v''+P(z)v=0\ \ \textrm{with}\ \ P(z)\geq c \Lambda z^2, \ \
 z\in \left(L\Lambda^{-\frac{1}{4}},(\ln
\Lambda)\Lambda^{-\frac{1}{4}}\right),
\end{equation}
with $c>0$ independent of both $\Lambda,L$. Unfortunately, it is not clear to us how to use the maximum principle to deduce the positivity and monotonicity
of $\textbf{v}_{2,in}$ in this remaining interval without too much effort. A possible way would be to show that $\textbf{v}_{2,in}'\left((\ln
\Lambda)\Lambda^{-\frac{1}{4}}\right)=w_{2,in}'\left((\ln
\Lambda)\Lambda^{-\frac{1}{4}}\right) <0$. This last task, however, would require us to keep track of the sharp super-exponential decay of the various
functions involved in the construction of $w_{2,in}$.
  Nevertheless, since
  \begin{equation}\label{eq54-}
  \textbf{v}_{2,in}\left((\ln
\Lambda)\Lambda^{-\frac{1}{4}}\right)=w_{2,in}\left((\ln
\Lambda)\Lambda^{-\frac{1}{4}}\right)\stackrel{(\ref{eqV1V2asympt}),(\ref{eqPhiHatAsympt})}{=}\Lambda^{-\frac{1}{4}}\mathcal{O}\left(e^{-c(\ln\Lambda)^2}\right)+\Lambda^{-\frac{3}{4}}\mathcal{O}\left(e^{-D(\ln\Lambda)}\right),
  \end{equation}
  for any $D>0$, as $\Lambda \to \infty$,
 we deduce   by Proposition \ref{proExistInn} (used mildly only at $z=\Lambda^{-\frac{1}{4}}$), relation (\ref{eqv2Lin})  and a   barrier
argument  that
\begin{equation}
\label{eqEmvo}
\left|\textbf{v}_{2,in}(z)\right|\leq C \Lambda^{-\frac{1}{4}} e^{-c\Lambda^{\frac{1}{2}}z^2}+\mathcal{O}(\Lambda^{-\infty}),
\ \ z\in \left[\Lambda^{-\frac{1}{4}}, (\ln \Lambda)\Lambda^{-\frac{1}{4}}\right].\end{equation}
In turn, by (\ref{eqv2Lin}) and a standard interpolation argument, we can easily  infer that
\begin{equation}\label{eqmonot2}
\textbf{v}_{2,in}'\left((\ln
\Lambda)\Lambda^{-\frac{1}{4}}\right)=\mathcal{O}(\Lambda^{-\infty})\ \ \textrm{as}\ \ \Lambda \to \infty.
\end{equation}
The above two estimates, and the analogous ones for $\textbf{v}_{1,in}$, will play a pivotal role in 'extending' $(\textbf{v}_{1,in},\textbf{v}_{2,in})$ to a heteroclinic solution to (\ref{eqEqGen})-(\ref{eqBdryGen}),
which can then easily be shown to have positive components with the right monotonicity properties.

\subsection{Adjusting the Hamiltonian constant of the inner genuine solution}\label{subsecAdjHamSolInn} From the calculations of
Subsection \ref{subsecAdjHam1} and Proposition \ref{proExistInn}, it follows that,
if $\alpha$ therein is inside $(0,1/4)$ and $\tilde{B}$ satisfies (\ref{eqBtildeConstr}), the equation for $\tilde{B}$
such that the Hamiltonian constant of the exact solution $(\textbf{v}_{1,in},\textbf{v}_{2,in})$
 of the inner problem  is
equal to $-\psi_0^2/2$ has the form
\[
2 \psi_0^2\tilde{B}\Lambda^\frac{1}{4}+\Lambda^{-1+\alpha}h(\Lambda,\tilde{B})=0,
\]
where the function $h$ is uniformly bounded in $\Lambda$ and continuous.
Consequently, by the Bolzano-Weistrass theorem, there exists at least one
\begin{equation}\label{eqBtilde}
\tilde{B}=\mathcal{O}\left(\Lambda^{-\frac{5}{4}+\alpha}\right)\ \ \textrm{as}\ \ \Lambda \to \infty,
\end{equation}
which satisfies the above equation ($\alpha\in (0,1/4)$ is still as in Proposition \ref{proExistInn}).
 Clearly, our working assumption (\ref{eqBtildeConstr}) is
satisfied  for $\alpha>0$ sufficiently small.
\section{Solution of the outer problem}\label{secSolOutReal}
In this section, we will construct a symmetric solution
$(\textbf{v}_{1,out},\textbf{v}_{2,out})$ to  system
(\ref{eqEqGen}) outside of the interval $I_\Lambda$ which, however,
agrees on $\partial I_\Lambda$ with the already constructed
solution $(\textbf{v}_{1,in},\textbf{v}_{2,in})$ of the inner
problem (in the $C^0$ sense) and satisfies the desired asymptotic
behaviour in (\ref{eqBdryGen}).
\subsection{A refined outer approximate solution $(w_{1,out},w_{2,out})$}\label{subsecRefOut} We first
consider a refinement of the outer approximate solution
$(v_{1,out},v_{2,out})$ that was constructed in Subsection
\ref{subsecOuter}, defined as
\[
w_{1,out}(z)=v_{1,out}(z)+\tilde{\tau}_1U'_1(z+\xi_1),\ \
w_{2,out}(z)=\textbf{v}_{2,in}\left((\ln
\Lambda)\Lambda^{-\frac{1}{4}} \right)\zeta(z),
\]
for $z\geq (\ln \Lambda)\Lambda^{-\frac{1}{4}}$; where
\[
\tilde{\tau}_1=-\frac{(v_{1,out}-\textbf{v}_{1,in})\left((\ln
\Lambda)\Lambda^{-\frac{1}{4}} \right)}{U'\left((\ln
\Lambda)\Lambda^{-\frac{1}{4}}+\xi_1
\right)}\stackrel{(\ref{eqUmonot}), (\ref{eqv1(out-in)}), \textrm{Prop.} \ref{proExistInn},
(\ref{eqBtilde})}{=}\mathcal{O}\left((\ln
\Lambda)\Lambda^{-\frac{5}{4}+\alpha} \right),
\]
as $\Lambda\to
\infty$
($\alpha \in (0,\frac{1}{4})$ still as in Proposition \ref{proExistInn}), and $\zeta \in
C_0^\infty(\mathbb{R})$ is a fixed cutoff function which is equal
to one on $[-1,1]$ (in this regard, keep in mind
(\ref{eqv2(out-in)}), (\ref{eqmonot2})).
Analogously we define $(w_{1,out},w_{2,out})$ for $z\leq -(\ln
\Lambda)\Lambda^{-\frac{1}{4}}$.

By construction, we have
\[
w_{i,out}\left(\pm(\ln \Lambda)\Lambda^{-\frac{1}{4}}
\right)=\textbf{v}_{i,in}\left(\pm(\ln \Lambda)\Lambda^{-\frac{1}{4}}
\right),\ \ i=1,2,
\]
and that the asymptotic behaviour (\ref{eqBdryGen}) is still satisfied.
\subsection{The remainder of the refined outer approximate solution}\label{subsecRemRefOut}
By recalling the proof of Lemma \ref{lemRemOut} and relations
(\ref{eqv2(out-in)}), (\ref{eqmonot2}), we find that the assertion
of the aforementioned lemma continues to hold for
$(w_{1,out},w_{2,out})$; except that in the  equations which were satisfied exactly now there is a remainder left, but  whose absolute value is bounded by a $\Lambda^{-\infty}$-small number
times a fixed, compactly supported function.

\subsection{The perturbation argument}\label{subsecPertOut}
We seek a solution of  system (\ref{eqEqGen}) as
\begin{equation}\label{eqcapt}\begin{array}{c}
 ({\textbf{v}_{1,out},\textbf{v}_{2,out}})=(w_{1,out},w_{2,out})+(\varphi_1,\varphi_2),
\ \ |z|\geq (\ln \Lambda)\Lambda^{-\frac{1}{4}},
 \\
    \\
  \varphi_i\left( \pm (\ln \Lambda)\Lambda^{-\frac{1}{4}}\right)=0,\ \
\lim_{z\to \pm \infty}\varphi_i(z)=0,\ \  i=1,2.
 \\
\end{array}
\end{equation}

Proceeding as in Subsection \ref{subsecPertIn}, we find that now
the corresponding linear operator is
\[
(\mathcal{L}+Q)\left(\begin{array}{c}
  \varphi_1 \\
  \\
  \varphi_2 \\
\end{array} \right)=\left(\begin{array}{c}
  -\varphi_1''+(3w_{1,out}^2-1)\varphi_1+\Lambda w_{2,out}^2 \varphi_1+2\Lambda w_{1,out}w_{2,out}\varphi_2 \\
    \\
  -\varphi_2''+(3w_{2,out}^2-1)\varphi_2+\Lambda w_{1,out}^2 \varphi_2+2\Lambda w_{1,out}w_{2,out}\varphi_1 \\
\end{array} \right).
\]
The invertibility properties of the above operator that we will
need are contained in the following proposition.
\begin{pro}\label{prolinOut}
Given $m>1$, there exist $\Lambda_2,C>0$ such that the boundary value problem
\begin{equation}\label{eqL+B}
(\mathcal{L}+Q)\left(\begin{array}{c}
  \varphi_1 \\
  \varphi_2 \\
\end{array} \right)=\left(\begin{array}{c}
  f_1 \\
  f_2 \\
\end{array}\right),
\ \ \varphi_i\left( (\ln \Lambda)\Lambda^{-\frac{1}{4}}
\right)=0,\ \ \lim_{z\to +\infty}\varphi_i(z)=0,\ \ i=1,2,
\end{equation}
where $f_1,f_2\in C\left[(\ln \Lambda)\Lambda^{-\frac{1}{4}},\infty
\right)$ decay  exponentially fast to zero,  has a unique
solution such that
\[
\|\varphi_1' \|_{L^\infty(J_\Lambda)}+\|\varphi_1
\|_{L^\infty(J_\Lambda)}\leq C \|f_1
\|_{L^\infty(J_\Lambda)}+\Lambda^{-m}\|f_2
\|_{L^\infty(J_\Lambda)},
\]
and \[ (\ln \Lambda)^{-1}\Lambda^{-\frac{1}{4}}\|\varphi_2'
\|_{L^\infty(J_\Lambda)}+\|\varphi_2  \|_{L^\infty(J_\Lambda)}\leq
\Lambda^{-m} \|f_1\|_{L^\infty(J_\Lambda)}+C (\ln
\Lambda)^{-2}\Lambda^{-\frac{1}{2}} \|f_2\|_{L^\infty(J_\Lambda)},
\] where
\[
J_\Lambda=\left((\ln \Lambda)\Lambda^{-\frac{1}{4}},\infty
\right),
\]
provided that $\Lambda\geq \Lambda_2$. An analogous estimate holds also in the negative outer region.
\end{pro}
\begin{proof}
As in Proposition \ref{proNonSymPro}, it is enough to establish the
validity of the asserted a-priori estimates. We note also that the continuous spectrum of $\mathcal{L}+Q$, when defined naturally in $L^2(J_\Lambda)\times L^2(J_\Lambda)$, coincides with the interval $[\Lambda^2-1,\infty)$, see \cite{alamaARMA15} or \cite{hislop}, which does not include zero by (\ref{eqHessian}).

We will first show that there exist constants $\Lambda_1,C>0$ such
that the following a-priori estimate holds: If $\phi \in
C^2(\bar{J_\Lambda})$ and $f \in C(\bar{J_\Lambda})$ satisfy
\[
  -\phi''+(3w_{1,out}^2-1)\phi+\Lambda w_{2,out}^2 \phi=f,\ \ z\in
  J_\Lambda,\]
  \[
 \phi\left((\ln \Lambda)\Lambda^{-\frac{1}{4}}\right)=0,
 \ \ \lim_{z\to \infty}\phi(z)=0, \] for
$\Lambda \geq \Lambda_1$, then
\[
\|\phi'  \|_{L^\infty(J_\Lambda)}+\|\phi
\|_{L^\infty(J_\Lambda)}\leq C \|f \|_{L^\infty(J_\Lambda)}.
\]
A preliminary observation is that, since
\begin{equation}\label{eqw2BoundOut}
\|w_{2,out}\|_{L^\infty(J_\Lambda)}=\mathcal{O}(\Lambda^{-\infty})\ \ \textrm{as}\ \ \Lambda \to \infty,
\end{equation}
it is clearly sufficient to show the a-priori estimate
\[
\|\phi \|_{L^\infty(J_\Lambda)}\leq C \|f
\|_{L^\infty(J_\Lambda)}.
\]
To this end, we will argue by contradiction. So, let us suppose
that there exist $\Lambda_n \to \infty$, $\phi_n \in
C^2(\overline{J_{\Lambda_n}})$ and $f_n \in
C(\overline{J_{\Lambda_n}})$ such that
\[
  -\phi_n''+(3w_{1,out}^2-1)\phi_n+\Lambda_n w_{2,out}^2 \phi_n=f_n,\ \ z\in
  J_{\Lambda_n},  \]
\[
 \phi_n\left((\ln \Lambda_n)\Lambda_n^{-\frac{1}{4}}\right)=0,
 \ \ \lim_{z\to \infty}\phi_n(z)=0, \]
while
\[
\|\phi_n \|_{L^\infty(J_{\Lambda_n})}=1\ \ \textrm{and}\ \ \|f_n
\|_{L^\infty(J_{\Lambda_n})}\to 0.
\]
Let
\[
\tilde{\phi}_n(z)=\phi_n\left(z+(\ln \Lambda_n)\Lambda_n^{-\frac{1}{4}} \right),\ \ \tilde{f}_n(z)=f_n\left(z+(\ln \Lambda_n)\Lambda_n^{-\frac{1}{4}} \right),
\ \ z\geq 0.
\]
Then, we have that
\[
  -\tilde{\phi}_n''+\left[3w_{1,out}^2\left(z+(\ln \Lambda_n)\Lambda_n^{-\frac{1}{4}} \right)-1\right]\tilde{\phi}_n+\Lambda_n w_{2,out}^2\left(z+(\ln \Lambda_n)\Lambda_n^{-\frac{1}{4}} \right) \tilde{\phi}_n=\tilde{f}_n,  \]
$z\in
  [0,\infty)$,
 $\tilde{\phi}_n(0)=0$,
$\lim_{z\to \infty}\tilde{\phi}_n(z)=0$,
while
\[
\|\tilde{\phi}_n \|_{L^\infty(0,\infty)}=1\ \ \textrm{and}\ \ \|\tilde{f}_n
\|_{L^\infty(0,\infty)}\to 0.
\]
Keeping in mind (\ref{eqw2BoundOut}), thanks again to standard
elliptic estimates and the usual diagonal argument,  passing to a
subsequence if necessary, and recalling the construction of
$w_{1,out}$,  we find that
\[
\tilde{\phi}_n \to \tilde{\phi}_\infty\ \ \textrm{in}\ \ C^1_{loc}[0,\infty),
\]
for some $\tilde{\phi}_\infty$ satisfying
\[
-\tilde{\phi}_\infty''+\left(3U_1^2(z)-1 \right)\tilde{\phi}_\infty=0,\ \ z>0;\ \ \tilde{\phi}_\infty(0)=0\ \ \textrm{and}\ \ \|\tilde{\phi}_\infty \|_{L^\infty(0,\infty)}\leq 1.
\]
Moreover, it is easy to see that $\tilde{\phi}_\infty$ is nontrivial since
\[
\left[3w_{1,out}^2-1+\Lambda_n w_{2,out}^2\right]\left(z+(\ln \Lambda_n)\Lambda_n^{-\frac{1}{4}} \right)\to 2,\   \textrm{as} \ z\to
\infty, \ \textrm{uniformly\ in} \ n,
\]
which implies that the points where $|\tilde{\phi}_n|$ attains its maximum
cannot escape at infinity. On the other hand, the first case in Lemma
\ref{lemUnondeg} implies that $\tilde{\phi}_\infty$ is identically equal to zero
which is a contradiction.

Applying the previously proven a-priori estimate to the first
equation of (\ref{eqL+B}), and recalling (\ref{eqw2BoundOut}), we
obtain that
\begin{equation}\label{eqcomb1}
\|\varphi_1' \|_{L^\infty(J_\Lambda)}+\|\varphi_1
\|_{L^\infty(J_\Lambda)}\leq C \|f_1
\|_{L^\infty(J_\Lambda)}+\mathcal{O}(\Lambda^{-\infty})\|\varphi_2
\|_{L^\infty(J_\Lambda)}.
\end{equation}

The situation in the second equation is considerably simpler.
Indeed, observing that
\begin{equation}\label{eqw1boundoutLower}
w_{1,out}(z)\geq c(\ln \Lambda)\Lambda^{-\frac{1}{4}},\ \ z\in
J_\Lambda,
\end{equation}
it follows easily that \[\begin{split} (\ln
\Lambda)^{-1}\Lambda^{-\frac{1}{4}}\|\varphi_2'
\|_{L^\infty(J_\Lambda)}+\|\varphi_2  \|_{L^\infty(J_\Lambda)}\leq \ \ \ \ \ \ \ \ \ \ \ \ \ \ \ \ \ \ \ \ \ \ \ \ \ \ \ \ \ \ \ \ \ \ \ \ \ \ \ \ \ \ \ \ \ \ \ \ \ \ \ \ \ \ \ \  & \\
\mathcal{O}( \Lambda^{-\infty}) \|\varphi_1\|_{L^\infty(J_\Lambda)}+C (\ln
\Lambda)^{-2}\Lambda^{-\frac{1}{2}} \|f_2\|_{L^\infty(J_\Lambda)}.
\end{split}\]

The assertion of the proposition, in the case of the positive outer region, now follows directly by combining
(\ref{eqcomb1}) and the above relation. The corresponding estimate in the negative outer region follows completely analogously.
\end{proof}

Armed with the above proposition and the observation made in
Subsection \ref{subsecRemRefOut}, concerning the remainder left by $(w_{1,out},w_{2,out})$, we can use the contraction
mapping principle to capture the desired $(\varphi_1,\varphi_2)$
in (\ref{eqcapt}) and arrive at the main result of this section.

\begin{pro}\label{proExistOut}
If $\Lambda>0$ is sufficiently large, system (\ref{eqEqGen}), for $|z|\geq (\ln
\Lambda)\Lambda^{-\frac{1}{4}}$, has a
solution $(\emph{\textbf{v}}_{1,out},\emph{\textbf{v}}_{2,out})$
of the form (\ref{eqcapt}) with
\[
\|\varphi_1' \|_{L^\infty\left((\ln
\Lambda)\Lambda^{-\frac{1}{4}},\infty\right)}+\|\varphi_1
\|_{L^\infty\left((\ln
\Lambda)\Lambda^{-\frac{1}{4}},\infty\right)}\leq C
(\ln \Lambda)^2\Lambda^{-\frac{3}{2}},
\]
\[
\|\varphi_2' \|_{L^\infty\left((\ln
\Lambda)\Lambda^{-\frac{1}{4}},\infty\right)}+\|\varphi_2
\|_{L^\infty\left((\ln
\Lambda)\Lambda^{-\frac{1}{4}},\infty\right)}=\mathcal{O}( \Lambda^{-\infty}),
\]
and the analogous estimate is valid for $ z \leq -(\ln
\Lambda)\Lambda^{-\frac{1}{4}}$.
\end{pro}

For future reference, we note that a standard barrier argument yields that
\begin{equation}
\label{eqEmvoOutt}
\left|\textbf{v}_{i,out}(z)\right|\leq \mathcal{O}(\Lambda^{-\infty}) e^{-c(\ln \Lambda)^\frac{1}{2}\Lambda^{\frac{1}{4}}|z|},
\ \ (-1)^i z\geq (\ln \Lambda)\Lambda^{-\frac{1}{4}},\ \ i=1,2.\end{equation}

\section{Existence and nondegeneracy of the heteroclinic orbit: proof of Theorems \ref{thmMain} and \ref{thmNonDeg}}\label{secMainRes}
In this section, we will prove our main result.
%Our main
%observation is the following.
So far, we have solved exactly the system (\ref{eqEqGen}) in the inner zone
$\left(-(\ln \Lambda)\Lambda^{-\frac{1}{4}},(\ln
\Lambda)\Lambda^{-\frac{1}{4}} \right)$ by
$(\textbf{v}_{1,in},\textbf{v}_{2,in})$, and in the outer zone
$|z|> (\ln
\Lambda)\Lambda^{-\frac{1}{4}}$ by its continuous extension
$(\textbf{v}_{1,out},\textbf{v}_{2,out})$.
Furthermore, the asymptotic behaviour in (\ref{eqBdryGen}) is
satisfied. In other words, the continuous and piecewise smooth
pair which is defined as
\begin{equation}\label{eqGenV1ap}
\textbf{v}_{i,ap}(z)=\left\{\begin{array}{ll}
  \textbf{v}_{i,in}(z), & |z|\leq (\ln \Lambda)\Lambda^{-\frac{1}{4}}, \\
  \textbf{v}_{i,out}(z), & |z|\geq (\ln \Lambda)\Lambda^{-\frac{1}{4}}, \\
\end{array} \right.\ \ \ i=1,2,
\end{equation}
is an exact solution to problem (\ref{eqEqGen})-(\ref{eqBdryGen})
with the exception of the two points $\pm (\ln
\Lambda)\Lambda^{-\frac{1}{4}}$. Loosely speaking, the above pair can be considered as a 'caricature' of the desired solution to (\ref{eqEqGen})-(\ref{eqBdryGen}) for large $\Lambda$.  As we will see next, this vague notion can    be made precise.
\subsection{A gluing argument: The global approximate solution $(\textbf{w}_{1,ap},\textbf{w}_{2,ap})$}\label{subsecGluing}
We recall that we have constructed the solution
$(\textbf{v}_{1,in},\textbf{v}_{2,in})$ of the inner problem
%without specific boundary conditions at $\pm (\ln
%\Lambda)\Lambda^{-\frac{1}{4}}$ but
such that its Hamiltonian constant is equal to $-\psi_0^2/2$,
which clearly is that of the aforementioned solutions of the
outer problems. The main observation is that this implies that the
jumps in the derivatives of $\textbf{v}_{1,ap},\textbf{v}_{2,ap}$ are
transcendentally small. Indeed, by the equality of the Hamiltonian
constants at $(\ln \Lambda)\Lambda^{-\frac{1}{4}}$, and the fact
that $\textbf{v}_{i,in}=\textbf{v}_{i,out}$, $i=1,2$, at that
point,   we have that
\[
\left[\textbf{v}_{1,in}'\right]^2+\left[\textbf{v}_{2,in}'\right]^2=\left[\textbf{v}_{1,out}'\right]^2+\left[\textbf{v}_{2,out}'\right]^2\
\ \textrm{at}\ \ (\ln \Lambda)\Lambda^{-\frac{1}{4}}.
\]
In turn, the corresponding estimates to (\ref{eq54-}), (\ref{eqmonot2}) for $\textbf{v}_{1,in}$, and Proposition \ref{proExistOut}, yield that
\[
\left[\textbf{v}_{1,in}'\right]^2-\left[\textbf{v}_{1,out}'\right]^2=\mathcal{O}\left(\Lambda^{-\infty} \right)
\ \textrm{at}\ \ (\ln \Lambda)\Lambda^{-\frac{1}{4}},\ \ \textrm{as}\ \ \Lambda \to \infty.
\]
Consequently, since
$
\textbf{v}_{1,in}'\left((\ln \Lambda)\Lambda^{-\frac{1}{4}} \right)\geq c$ and $\textbf{v}_{1,out}'\left((\ln \Lambda)\Lambda^{-\frac{1}{4}} \right)\geq c$
(keep in mind Propositions \ref{proExistInn} and \ref{proExistOut}), we infer that
\[
\left(\textbf{v}_{1,in}'-\textbf{v}_{1,out}'\right)\left((\ln
\Lambda)\Lambda^{-\frac{1}{4}}\right)=\mathcal{O}\left(\Lambda^{-\infty} \right),\ \ \textrm{as}\ \ \Lambda \to \infty,
\]
as desired. Naturally, the proof of the analogous property  for the second components at $-(\ln
\Lambda)\Lambda^{-\frac{1}{4}}$ is completely analogous.

We are now ready to define our global  $C^1$-smooth approximate solution
to  problem (\ref{eqEqGen})-(\ref{eqBdryGen}) as
\begin{equation}\label{eqGenW1apg}
\textbf{w}_{i,ap}(z)=\textbf{v}_{i,ap}(z)+(s_i)_- e^{-\Lambda^{\frac{1}{4}}\left|z+(\ln
\Lambda)\Lambda^{-\frac{1}{4}}\right|} +(s_i)_+ e^{-\Lambda^{\frac{1}{4}}\left|z-(\ln
\Lambda)\Lambda^{-\frac{1}{4}}\right|},\end{equation}
 where the numbers  \begin{equation}\label{eqQuasispm}(s_i)_\pm=\mathcal{O}\left(\Lambda^{-\infty} \right)\ \ \textrm{as}\ \ \Lambda \to \infty, \ \ i=1,2,\end{equation} are
chosen so that $\textbf{w}_{i,ap}$, $i=1,2$, are
$C^1$ at $\pm(\ln \Lambda)\Lambda^{-\frac{1}{4}}$.
  This approximate solution leaves a remainder in (\ref{eqEqGen}) which is uniformly of order $\mathcal{O}\left(\Lambda^{-\infty} \right)$ (keep in mind, however, that it may have finite jump discontinuities at the two gluing points), while the asymptotic behaviour (\ref{eqBdryGen}) as $z\to \pm \infty$ is fulfilled exactly.

\subsection{Perturbing the global approximate solution to a genuine one}

Even though $(\textbf{w}_{1,ap},\textbf{w}_{2,ap})$ is an extremely good approximate solution, perturbing it to a genuine one by some type of local inversion argument is a subtle task. Indeed, the associated linearized operator
 \begin{equation}\label{eqMcal}
\mathcal{M} \left(\begin{array}{c}
  \varphi_1 \\
  \\
  \varphi_2 \\
\end{array} \right)=\left(\begin{array}{c}
  -\varphi_1''+(3\textbf{w}_{1,ap}^2-1)\varphi_1+\Lambda \textbf{w}_{2,ap}^2 \varphi_1+2\Lambda \textbf{w}_{1,ap}\textbf{w}_{2,ap}\varphi_2 \\
    \\
  -\varphi_2''+(3\textbf{w}_{2,ap}^2-1)\varphi_2+\Lambda \textbf{w}_{1,ap}^2 \varphi_2+2\Lambda \textbf{w}_{1,ap}\textbf{w}_{2,ap}\varphi_1 \\
\end{array} \right)
\end{equation}
 is nearly non-invertible because $(\textbf{w}_{1,ap}',\textbf{w}_{2,ap}')$ is  extremely close to being in the kernel.
Nevertheless,  we will overcome this difficulty by adapting to our setting a well known \emph{variational Lyapunov-Schmidt method} (see \cite{del2008giorgi} and the references therein).

 Naturally, we seek a solution of (\ref{eqEqGen})-(\ref{eqBdryGen}) in the form  \begin{equation}\label{eqdPert}(v_1,v_2)=(\textbf{w}_{1,ap},\textbf{w}_{2,ap})+(\varphi_1,\varphi_2),\end{equation} with fluctuations
satisfying the orthogonality condition
\begin{equation}\label{eqQuasiOrthNew2}
\int_{-\infty}^{\infty}\left(\textbf{w}_{1,ap}'\varphi_1+\textbf{w}_{2,ap}'\varphi_2 \right)dz=0.
\end{equation}
The following proposition, concerning the so-called \emph{linear projected problem}, makes it legitimate to apply the aforementioned Lyapunov-Schmidt method.
\begin{pro}\label{proGenLinearQ}
Given $\beta>0$, there exist constants $\Lambda_3,C>0$ such that if $\Lambda\geq \Lambda_3$ and $(h_1,h_2)\in L^2(\mathbb{R})\times L^2(\mathbb{R})$
with $\|(h_1,h_2)\|_*<\infty$, the problem
\begin{equation}\label{eqQuasiEq}
\mathcal{M}\left(\begin{array}{c}
  \phi_1 \\

  \phi_2 \\
\end{array} \right)=\left(\begin{array}{c}
 h_1 \\

 h_2 \\
\end{array} \right)+c_\Lambda \left(\begin{array}{c}
 \textbf{w}_{1,ap}' \\

 \textbf{w}_{2,ap}' \\
\end{array} \right),
\end{equation}
where the linear operator $\mathcal{M}$ is as in (\ref{eqMcal}), has a unique solution
$(\phi_1,\phi_2)\in H^2(\mathbb{R})\times H^2(\mathbb{R})$ and $c_\Lambda \in \mathbb{R}$ such that
\begin{equation}\label{eqQuasiOrth}
\int_{-\infty}^{\infty}\left(\textbf{w}_{1,ap}'\phi_1+\textbf{w}_{2,ap}'\phi_2 \right)dz=0
\end{equation}
and
\begin{equation}\label{eqGenAprQuasiNovel}
 \sum_{i=1}^{2}\|\phi_i\|_{L^\infty(\mathbb{R})}
  \leq C\Lambda^\frac{\beta}{2}\|(h_1,h_2)\|_*,
\end{equation}
where $\| \cdot \|_*$ stands for the weighted norm
\[
\|(h_1,h_2)\|_*=\sum_{i=1}^{2}\| \left( \Lambda^\frac{1+\beta}{2} |z|^{2+2\beta}+1 \right)h_i\|_{L^\infty(\mathbb{R})}.
\]
\end{pro}
\begin{proof}
The proof will be divided into three steps.

\textbf{Step 1.} We will first establish the validity of the stronger a-priori estimate \begin{equation}\label{eqGenAprQuasi}
\sum_{i=1}^{2}\|\phi_i\|_{L^\infty(\mathbb{R})}
\leq C\Lambda^{-\frac{1}{2}}\|(h_1,h_2)\|_*,
\end{equation} when the constant $c_\Lambda$ in
(\ref{eqQuasiEq}) is equal to zero. To this end, as usual, we will argue by contradiction. So, as in Proposition \ref{proNonSymPro}, let us suppose that there are $\Lambda_n \to \infty$,
$(\phi_{1,n},\phi_{2,n})\in H^2(\mathbb{R})\times H^2(\mathbb{R})$, $(h_{1,n},h_{2,n})\in L^2(\mathbb{R})\times L^2(\mathbb{R})$ satisfying
\begin{equation}\label{eqNonEq}\left\{
	\begin{array}{c}
	-\phi_{1,n}''+(3\textbf{w}_{1,ap}^2-1)\phi_{1,n}+\Lambda_n \textbf{w}_{2,ap}^2 \phi_{1,n}+2\Lambda_n \textbf{w}_{1,ap}\textbf{w}_{2,ap}\phi_{2,n}= h_{1,n}, \\
	\\
	-\phi_{2,n}''+(3\textbf{w}_{2,ap}^2-1)\phi_{2,n}+\Lambda_n \textbf{w}_{1,ap}^2 \phi_{2,n}+2\Lambda_n \textbf{w}_{1,ap}\textbf{w}_{2,ap}\phi_{1,n}= h_{2,n}, \\
	\end{array}\ \ z\in \mathbb{R},
	\right.
	\end{equation}
	and
	\begin{equation}\label{eqNonOrth}
	\int_{-\infty}^{\infty}\left(\textbf{w}_{1,ap}' \phi_{1,n}+\textbf{w}_{2,ap}' \phi_{2,n} \right)dz=0,
	\end{equation} while
\begin{equation}\label{eqQuasiContra}
\|\phi_{1,n}\|_{L^\infty(\mathbb{R})}=1,\ \ \|\phi_{2,n}\|_{L^\infty(\mathbb{R})}\leq 1,\ \  \Lambda_n^{-\frac{1}{2}}\sum_{i=1}^{2}\| \left( \Lambda_n^\frac{1+\beta}{2} |z|^{2+2\beta}+1 \right)h_{i,n}\|_{L^\infty(\mathbb{R})}\to 0.
\end{equation}

We will first show  that
\begin{equation}\label{eqQuasiAlg}
\left|\phi_{i,n}(z) \right|\leq 2 \frac{\|\phi_{i,n}\|_{L^\infty(\mathbb{R})}}{1+\Lambda_n^\frac{1+\beta}{2} |z|^{2+2\beta}},\ \ (-1)^iz\geq \Lambda_n^{-\frac{1}{4}},\ \ i=1,2.
\end{equation}
To this end, the main observation is that, owing to the first equation in (\ref{eqNonEq}),     $\phi_{1,n}$ satisfies a linear inhomogeneous equation of the form
\begin{equation}\label{eqEval11}
-\phi_1''+p(z)\phi_1=f(z),\ \ z\leq 0,
\end{equation}
with
\begin{equation}\label{eqEval22}
p(z)\geq c \Lambda_n^{\frac{1}{2}}\ \ \textrm{and}\ \ |f(z)|\leq C\frac{\Lambda_n^\frac{1}{2}}{1+\Lambda_n^\frac{1+\beta}{2} |z|^{2+2\beta}},\ \ z\leq 0,
\end{equation}
(as usual, the generic constants $c,C>0$ do not depend on $n$).
The above relation follows readily from the already established properties of $\textbf{w}_{1,ap},
\textbf{w}_{2,ap}$ and (\ref{eqQuasiContra}).
Let us just point out that special attention should be paid in showing that
\[
\left|\textbf{w}_{1,ap}
\textbf{w}_{2,ap} \right|\leq C\Lambda_n^{-\frac{1}{2}} e^{c\Lambda_n^\frac{1}{4}z},\ \ z\leq -\Lambda_n^{-\frac{1}{4}},
\]
which essentially follows by combining the estimate
\[
c\Lambda^{-\frac{1}{4}} \leq \textbf{w}_{2,ap}(z)\leq C(\Lambda^{-\frac{1}{4}}+|z|), \ \  z\leq 0,
\]
with the analog of (\ref{eqEmvo}) for $\textbf{v}_{1,in}$, and (\ref{eqEmvoOutt}).
Analogous considerations apply to $\phi_{2,n}$.
The desired estimate (\ref{eqQuasiAlg}) follows from (\ref{eqEval11})-(\ref{eqEval22}), the corresponding relations for $z\geq 0$, and a barrier argument (see also \cite[pg. 435]{herve1994etude}).

With this preliminary step, we can now return to showing that a contradiction occurs.
Using (\ref{eqNonEq}),  (\ref{eqQuasiContra}), (\ref{eqQuasiAlg}), together with standard elliptic estimates and the
	familiar diagonal argument, passing to a subsequence if necessary, we find that
	\begin{equation}\label{eqNonPhinToInfty}
	\phi_{i,n}\to \phi_{i,\infty} \ \ \textrm{in}\ \ C^1_{loc}\left(\mathbb{R}\setminus \{ 0\}\right),
	\end{equation}
	where \begin{equation}\label{eqNonPhiApei00}\phi_{i,\infty}(z)=0\ \ \textrm{for}\ \ (-1)^iz>0,\end{equation} while
	\[
	-\phi_{i,\infty}''+\left(3U_i^2(z)-1 \right)\phi_{i,\infty}=0,\ \ (-1)^iz<0,\ \ i=1,2.
	\]
	Since $\phi_{i,\infty}$ is bounded for $z\neq 0$, by Lemma \ref{lemUnondeg} we must have that
	\begin{equation}\label{eqNonPhiApeiU'}\phi_{i,\infty}(z)=a_i U_i'(z),\ \ (-1)^iz<0,\ \ i=1,2,
	\end{equation}
for some $a_1,a_2\in \mathbb{R}$	(the above $a_1,a_2$ are not  related to those in Proposition \ref{proNonSymPro}).
	 Moreover, passing to the limit in the orthogonality relation (\ref{eqQuasiOrth}),  with the help of
Lebesgue's dominated convergence theorem, we obtain that
\begin{equation}\label{eqOrthLebs}
a_1 \int_{0}^{\infty}(U_1')^2dz+a_2\int_{-\infty}^{0}(U_2')^2dz=0.
\end{equation}

Next, motivated by the proof of Proposition \ref{proNonSymPro}, we wish to show that
\begin{equation}\label{eqNonDegInter}
\left|\phi_{i,n}-a_iU_i' \right|\leq \frac{C}{1+\Lambda_n^\frac{\beta}{2}|z|^{2\beta}}+o(1),\ \ i=1,2,
\end{equation}
uniformly for $0\leq (-1)^{i+1}z\leq 1$, as $n\to \infty$.
We will show this only for $i=1$, as the case $i=2$ can be handled similarly.
To this end, dropping for the moment some of the subscripts $n$ to relax the notation, we let
	\begin{equation}\label{eqdrop}
	\psi_1=\phi_1-a_1 U_1'\to 0\ \ \textrm{in} \ \ C^1_{loc}\left(0,\infty\right),
	\end{equation}
	and observe that
	\[
	\begin{array}{rcl}
	-\psi_1''+(3\textbf{w}_{1,ap}^2-1)\psi_1+\Lambda_n \textbf{w}_{2,ap}^2 \psi_1 & = & -2\Lambda_n \textbf{w}_{1,ap}\textbf{w}_{2,ap}\phi_2+h_1-a_1\Lambda_n \textbf{w}_{2,ap}^2U_1'
	\\
	&   &   \\
	&   & -3a_1(\textbf{w}_{1,ap}^2-U_1^2)U_1',
	\end{array}
	\]
	for $z\geq 0$.
It is convenient to write
	\begin{equation}\label{eqNonDecomp}
	\psi_1(z)=\tilde{\psi}_1(z)+\hat{\psi}_1(z) \ \ \textrm{for}\ \ z\in [0,1],
	\end{equation}
	where $\tilde{\psi}_1$ is the unique solution of the following boundary value problem:
	\[
	\left\{\begin{array}{ll}
	\tilde{\psi}_1''=(3\textbf{w}_{1,ap}^2-1)\psi_1+3a_1(\textbf{w}_{1,ap}^2-U_1^2)U_1', & z\in (0,1), \\
	&   \\
	\tilde{\psi}_1( 0)=\tilde{\psi}_1( 1)=0. &
	\end{array}
	\right.
	\]
Concerning $\tilde{\psi}_1$, we observe that standard elliptic estimates imply that
	\[
	\| \tilde{\psi}_1 \|_{H^2(0,1)}\leq C  \|(3\textbf{w}_{1,ap}^2-1)\psi_1+3a_1(\textbf{w}_{1,ap}^2-U_1^2)U_1'\|_{L^2(0,1)}\to 0,
	\]
	where the last limit holds by virtue of (\ref{eqdrop}), the construction of $\textbf{w}_{1,ap}$ and Lebesgue's dominated convergence theorem. Therefore,
	\begin{equation}\label{eqNonC1tilde}
	\| \tilde{\psi}_1 \|_{C^1(0,1)}\to 0.
	\end{equation}
Thus, in view of (\ref{eqdrop}) and (\ref{eqNonC1tilde}), we get that the function $\hat{\psi}_1$ in the decomposition (\ref{eqNonDecomp}) satisfies
	\[
	\hat{\psi}_1\to 0\ \ \textrm{in}\ \ C^1_{loc}(0,1].
	\]
	Then, since $\hat{\psi}_1$ satisfies
	\[
	\hat{\psi}_1''=\Lambda_n \textbf{w}_{2,ap}^2 \psi_1 +2\Lambda_n \textbf{w}_{1,ap}\textbf{w}_{2,ap}\phi_2-h_1+a_1\Lambda_n \textbf{w}_{2,ap}^2U_1',\ \ z\in (0,1),
	\]
similarly	   as in Proposition \ref{proNonSymPro} (also keep in mind the derivation of (\ref{eqEval22})), we obtain that
	\[
	|\hat{\psi}_1(z)|\leq \frac{C}{1+\Lambda_n^\frac{\beta}{2}|z|^{2\beta}}+o(1),\ \ \textrm{uniformly on}\ \ [0,1],\ \ \textrm{as}\ \ n\to \infty.
	\]
The desired estimate (\ref{eqNonDegInter}) now follows immediately from (\ref{eqdrop}), (\ref{eqNonDecomp}), (\ref{eqNonC1tilde}) and the above relation.

On the other side, similarly to the proof of Proposition \ref{proNonSymPro}, passing to a further subsequence if needed, we  get that \begin{equation}\label{eqNonDegInnerB}
	\phi_i(\Lambda_n^{-\frac{1}{4}}x)\to \texttt{b} V'_i(x) \ \ \textrm{in} \ \ C^1_{loc}(\mathbb{R}),\ \ i=1,2,
	\end{equation}
for some $\texttt{b}\in \mathbb{R}$.

By the same arguments that led to (\ref{eqa1b})-(\ref{eqa2b}), we find from (\ref{eqNonDegInter})
and (\ref{eqNonDegInnerB}) that
 \[
 a_1=\texttt{b},\ \ a_2=\texttt{b}.
 \]
 Hence, in view of (\ref{eqOrthLebs}), we get that
 \[a_1=a_2=\texttt{b}=0.\]
Then, by combining (\ref{eqQuasiAlg}), (\ref{eqNonPhinToInfty}), (\ref{eqNonPhiApei00}), (\ref{eqNonPhiApeiU'}), (\ref{eqNonDegInter}) and (\ref{eqNonDegInnerB}), we obtain that
\[
\|\phi_{i,n}\|_{L^\infty(\mathbb{R})}\to 0,\ \ i=1,2,
\]
(note also that, as in Proposition \ref{prolinOut}, the points where $|\phi_{i,n}|$ achieves its maximum cannot escape at infinity).
The above relation contradicts the first relation in (\ref{eqQuasiContra}), and completes the proof of Step 1.

\textbf{Step 2.} We will show that the a-priori estimate (\ref{eqGenAprQuasi}) holds for the full problem (\ref{eqQuasiEq})-(\ref{eqQuasiOrth}).
Testing (\ref{eqQuasiEq}) by $(\textbf{w}_{1,ap}',\textbf{w}_{2,ap}')$ gives that
\begin{equation}\label{eqQuasiNorm}
|c_\Lambda|\leq C \|(h_1,h_2)\|_{*}+C\left|\left<\mathcal{M}\left(\begin{array}{c}
  \phi_1 \\

  \phi_2 \\
\end{array} \right),\left(\begin{array}{c}
 \textbf{w}_{1,ap}' \\

 \textbf{w}_{2,ap}' \\
\end{array} \right) \right>_{L^2(\mathbb{R})\times L^2(\mathbb{R})} \right|,
\end{equation}
where we also used that \begin{equation}
\label{eqchocY}
|\textbf{w}_{i,ap}'|\leq Ce^{-c|z|},\ \  z\in\mathbb{R},\ \ i=1,2.\end{equation}
Unfortunately, since $\textbf{w}_{1,ap}',\textbf{w}_{2,ap}'$ are merely in $H^1(\mathbb{R})$, we cannot use directly the self-adjointness of $\mathcal{M}$
in the last term of the above relation to exploit that
\[
\|\mathcal{M}\left(\begin{array}{c}
 \textbf{w}_{1,ap}' \\

 \textbf{w}_{2,ap}' \\
\end{array} \right)   \|_{L^1(\mathbb{R}\setminus \{\pm (\ln\Lambda)\Lambda^{-\frac{1}{4}}\})\times L^1(\mathbb{R}\setminus \{\pm (\ln\Lambda)\Lambda^{-\frac{1}{4}}\})}=\mathcal{O}(\Lambda^{-\infty}).
\]
 Nevertheless, from  (\ref{eqGenW1apg})-(\ref{eqQuasispm}) and the fact that
\[
\textbf{v}_{i,ap}''\left(\left[\pm (\ln\Lambda)\Lambda^{-\frac{1}{4}} \right]^- \right)=\textbf{v}_{i,ap}''\left(\left[\pm (\ln\Lambda)\Lambda^{-\frac{1}{4}} \right]^+ \right)
\]
(keep in mind (\ref{eqGenV1ap})), we find that the jumps of $\textbf{w}_{i,ap}''$ at $\pm (\ln\Lambda)\Lambda^{-\frac{1}{4}} $
are of order $\mathcal{O}(\Lambda^{-\infty})$ as $\Lambda \to \infty$, $i=1,2$.
Hence, splitting the integral under consideration into three parts, integrating each one by parts using the self-adjointness of $\mathcal{M}$ and the previous observation to estimate the boundary terms, %keeping in mind the discussion in Subsubsection \ref{subsubsecRemGlob},
 we reach the bound
\[
\left<\mathcal{M}\left(\begin{array}{c}
  \phi_1 \\

  \phi_2 \\
\end{array} \right),\left(\begin{array}{c}
 \textbf{w}_{1,ap}' \\

 \textbf{w}_{2,ap}' \\
\end{array} \right) \right>_{L^2(\mathbb{R})\times L^2(\mathbb{R})} =\mathcal{O}(\Lambda^{-\infty})\sum_{i=1}^{2}\|\phi_i\|_{L^\infty(\mathbb{R})} \ \ \textrm{as}\ \Lambda \to \infty.
\]
In turn, via (\ref{eqQuasiNorm}), we obtain that
\[
|c_\Lambda|\leq C \|(h_1,h_2)\|_*+\mathcal{O}(\Lambda^{-\infty})\sum_{i=1}^{2}\|\phi_i\|_{L^\infty(\mathbb{R})}.
\]
On the other hand, applying the conclusion of Step 1 to (\ref{eqQuasiEq})-(\ref{eqQuasiOrth}), and keeping in mind (\ref{eqchocY}), we get that
\[
\sum_{i=1}^{2}\|\phi_i\|_{L^\infty(\mathbb{R})}\leq C\Lambda^{-\frac{1}{2}}\|(h_1,h_2)\|_*+C\Lambda^\frac{\beta}{2} |c_\Lambda|.
\]
The desired a-priori estimate now follows at once by combining the above two relations.
We point out that the main reason for using a power weight instead of the more convenient exponential one, as in Proposition  \ref{proNonSymPro} and Corollary \ref{corlinBasic}, was in order to get an efficient estimate  in the last term of the above relation.

\textbf{Step 3.} We will establish the existence of a unique solution $(\phi_1,\phi_2)\in H^2(\mathbb{R})\times H^2(\mathbb{R})$ and $c_\Lambda \in \mathbb{R}$
to  problem (\ref{eqQuasiEq})-(\ref{eqQuasiOrth}), given $(h_1,h_2)$ as in the assertion of the proposition.
Let $\mathcal{X}$ denote the subspace of $H^2(\mathbb{R})\times H^2(\mathbb{R})$ which consists of pairs $\Phi=(\phi_1,\phi_2)$ satisfying the orthogonality condition (\ref{eqQuasiOrth}).
%Considering the following bilinear form in $\mathcal{X}$:
%\[
%\mathcal{B}(\Phi,\Psi)=\left<\mathcal{M}(\Phi),\Psi\right>_{L^2(\mathbb{R})\times L^2(\mathbb{R})},
%\]
The problem (\ref{eqQuasiEq})-(\ref{eqQuasiOrth}) admits the following weak formulation: find $\Phi \in \mathcal{X}$ such that
\[
\left<\mathcal{M}(\Phi),\Psi\right>_{L^2(\mathbb{R})\times L^2(\mathbb{R})}=\left<H,\Psi\right>_{L^2(\mathbb{R})\times L^2(\mathbb{R})}\ \ \ \forall \ \Psi \in \mathcal{X},
\]
(where $H=(h_1,h_2)$). This weak formulation can then be readily put in the operator form
\[
\mathbb{M}(\Phi)=\hat{H},
\]
where $\mathbb{M}:\mathcal{X}\to \mathcal{X}$ is self-adjoint, and $\hat{H} \in \mathcal{X}$ depends linearly on $H$.
The a-priori estimate of Step 2 implies that, for $\hat{H}=0$, there is only the trivial solution. Consequently, by the self-adjoint property of $\mathbb{M}$,
we infer that the above problem has a solution $\Phi \in \mathcal{X}$ (see also \cite[Lem. 4.2]{lund2015existence}), which is clearly unique. This completes the proof of Step 3 and also of the proposition.
\end{proof}

Armed with the above proposition, we can apply the contraction mapping theorem in these weighted spaces to show that the   \emph{nonlinear projected problem}	
\[
	\left\{
	\begin{array}{c}
		-v_1''+v_1^3- v_1+\Lambda v_2^2 v_1=c_\Lambda \mathbf{w}_{1,ap}', \\
		\\
		-v_2''+v_2^3-v_2+\Lambda v_1^2 v_2=c_\Lambda \mathbf{w}_{2,ap}',
	\end{array}
	\right.
\]
has a solution $(v_1,v_2)$ and $c_\Lambda$ such that
\begin{equation}\label{eqviProj}
v_i=\mathbf{w}_{i,ap}+\varphi_i \ \ \textrm{with}\ \ \varphi_i \in H^2(\mathbb{R})\ \ \textrm{and}\ \
\|\varphi_i\|_{L^\infty(\mathbb{R})}=\mathcal{O}(\Lambda^{-\infty})\ \ \textrm{as} \ \ \Lambda \to \infty,\ \ i=1,2.
\end{equation}
Moreover, the fluctuation $(\varphi_1,\varphi_2)$ satisfies the orthogonality condition (\ref{eqQuasiOrthNew2}),
while the constant $c_\Lambda$ is of order $\mathcal{O}(\Lambda^{-\infty})$ as $\Lambda \to \infty$.
 Then,  elliptic regularity theory  imply that the solution is smooth (up to this moment, we know that $v_1',v_2' \in H^2(\mathbb{R})$). To this end, testing the above nonlinear projected problem with $(v_1',v_2')$, thanks to the gradient structure in the lefthand side, yields that
\[\begin{array}{ccc}
    0 & = & c_\Lambda\sum_{i=1}^{2}\int_{-\infty}^{\infty}\left[\left(\mathbf{w}_{i,ap}' \right)^2+\mathbf{w}_{i,ap}'\varphi_i'\right]dz \\
      &   &   \\
      & = & c_\Lambda\sum_{i=1}^{2}\int_{-\infty}^{\infty}\left[\left(\mathbf{w}_{i,ap}' \right)^2-\mathbf{w}_{i,ap}''\varphi_i\right]dz.
  \end{array}
\]
In turn, using the rough estimates
\[
\int_{-\infty}^{\infty}\left(\mathbf{w}_{i,ap}' \right)^2dz\geq c,\ \ \left|\mathbf{w}_{i,ap}''(z)\right|\leq C\Lambda^\frac{1}{4}e^{-c|z|},\ z\in \mathbb{R},\ i=1,2,
\]
and (\ref{eqviProj}), we can conclude that $c_\Lambda=0$ for $\Lambda$ sufficiently large, as desired.

\subsection{Proof of Theorem \ref{thmMain}}
\subsubsection{The estimates} The asserted estimates in Theorem \ref{thmMain} follow readily by taking into account the construction of the various approximate solutions, the estimates in Propositions \ref{proExistInn}, \ref{proExistOut}, and (\ref{eqviProj}) (the latter relation can  be differentiated once in the natural way). In particular, for the decay estimates (\ref{eqEmvoRRR}) and (\ref{v1expdecay}), keep in mind (\ref{eqEmvo}) and (\ref{eqEmvoOutt}) respectively.
\subsubsection{Positivity, monotonicity and decay properties of the heteroclinic orbit}\label{subsubMonot}
Armed with the previously proven $C^1$-uniform estimates for the constructed heteroclinic solution $(v_1,v_2)$,
we can complete the qualitative analysis of Subsubsection \ref{subsubsecMonotIn}, and thus the  proof of Theorem \ref{thmMain}.

The main observation is that $v_2$
satisfies a linear equation of the form
\[
-v''+\tilde{P}(z)v=0,\ \ z\geq L \Lambda^{-\frac{1}{4}},
\]
with
\[
\tilde{P}(z)\geq \left\{\begin{array}{ll}
c\Lambda z^2, & z\in (L \Lambda^{-\frac{1}{4}},\delta), \\
&   \\
c \Lambda,& z\in [\delta,\infty),
\end{array}
\right.
\]
for some fixed small $\delta>0$ (keep in mind (\ref{eqv2Lin})).
Hence, since $v_2(L \Lambda^{-\frac{1}{4}})\geq c \Lambda^{-\frac{1}{4}}$, $v_2'(L \Lambda^{-\frac{1}{4}})\leq -c$ for large $\Lambda$
(recall (\ref{eqv2posIn})) and $\lim_{z\to \infty}v_2(z)=0$, we deduce by the maximum principle that
\[
v_2>0\ \  \textrm{and}\ \ v_2'<0\ \ \textrm{on} \ \ [L  \Lambda^{-\frac{1}{4}},\infty).\]
In summary, so far we have shown that
\begin{equation}\label{eqv2analog}
v_2>0\ \ \textrm{and}\ \ v_2'<0\ \ \textrm{on}\ \ \left[-(\ln \Lambda) \Lambda^{-\frac{1}{4}},\infty \right).
\end{equation}
In fact, by the use of barriers and standard elliptic estimates, it follows readily that
\begin{equation}\label{eqdecay14}
v_2(z)-\Lambda^{-\frac{1}{4}}v_2'(z)\leq C \Lambda^{-\frac{1}{4}} e^{-c\Lambda^\frac{1}{4}z},\ \ z\geq 0.
\end{equation}
Moreover, the above estimate can be improved for large $z$: Given any fixed $d>0$,
\[
v_2(z)-\Lambda^{-\frac{1}{2}}v_2'(z)\leq Cv_2(d)e^{-c\Lambda^\frac{1}{2}z},\ \ z\geq d.
\]

The previously proven $C^1$-uniform estimates for the convergence of $v_2$ to $U_2$ over $\left(-\infty,-(\ln \Lambda)\Lambda^{-\frac{1}{4}} \right]$
guarantee that the same holds on any fixed interval of the form $[-M,\infty)$, provided that $\Lambda$ is sufficiently large. In particular,
$v_2(-M)\to U_2(-M)$ and $v_2'(-M)\to U_2'(-M)<0$ as $\Lambda \to \infty$. To conclude that $v_2$ is still decreasing in $(-\infty,-M)$, it is enough to apply
the maximum principle to the linear equation that is satisfied by $v_2'$. Indeed, using the analogous property to (\ref{eqv2analog}) for $v_1$, and our previous observations at $-M$, we find that the function $\psi \equiv v_2'$ satisfies
\[
-\psi''+\left(3v_2^2-1+\Lambda v_1^2 \right)\psi=-2\Lambda v_2 v_1 v_1'\leq 0,\ \ z\leq -M;\ \ \psi(-\infty)=0,\  \psi(-M)<0,
\]
with $3v_2^2-1+\Lambda v_1^2>0$ on $(-\infty,-M]$ (having increased the value of $M$ if necessary).
%  Actually, by the use of barriers and standard elliptic estimates in the equation for the %auxilliary function $\rho=v_1^2+(1-v_2)^2$, we readily find that
%\begin{equation}\label{eqdecay2}
%1-v_2(z)-\Lambda^{-\frac{1}{2}}v_2'(z)\leq Ce^{c\Lambda^\frac{1}{2}z},\ \ z\leq -M,
%\end{equation}
%(having increased the value of $M$ if necessary).

Analogously we argue for showing that
\[
 v_1'>0\ \ \textrm{in}\ \ \mathbb{R}.
\]
\begin{rem}\label{remSmooth}
A careful inspection of the proofs reveals that the solution provided by Theorem \ref{thmMain} depends smoothly on $\Lambda$ since there is a version of the contraction mapping theorem for operators depending  on parameters.
\end{rem}

\begin{rem}An effective approach for constructing heteroclinic orbits in singularly perturbed systems of ordinary differential equations is to make use of
\emph{geometric singular perturbation theory} (see \cite{kuhen} and the references therein). In particular, at least heuristically,  the blow-up   problem (\ref{eqBUsystem}) brings to mind the
recent  blow-up approach to this theory,  used to deal with problems involving loss of normal hyperbolicity
(see \cite{schecter2010heteroclinic} and the references therein). However,
we have not been able to put system (\ref{eqEqGen}) in the slow-fast form that is required for the aforementioned machinery to apply. In any case, we believe that the approach of the current paper extends in a natural way to deal with analogous problems in the broader context of elliptic systems of partial differential equations.\end{rem}

\subsection{Nondegeneracy of the heteroclinic: Proof of Theorem \ref{thmNonDeg}}

\begin{proof}[Proof of Theorem \ref{thmNonDeg}]
It has been shown in \cite{alamaARMA15} that  the lowest point in the spectrum of $\textbf{M}$ is $0$ which is a simple eigenvalue
with $(v_1',v_2')$ as the associated eigenfunction.  It was also shown therein that the
continuous spectrum of $\textbf{M}$ coincides with the interval $[\Lambda^2-1,\infty)$. So, it is enough to show that the second eigenvalue $\mu>0$ of $\textbf{M}$ (should it exist) is bounded away from $0$ independently
of large $\Lambda$. To this end, we will argue by contradiction.

Suppose, to the contrary, that there are $\Lambda_n\to \infty$ such that the second eigenvalue $\mu_n>0$ of $\textbf{M}$ exists and satisfies
\[
\mu_n\to 0.
\]
Hence,  there would exist an associated eigenfunction  $(\varphi_{1,n}, \varphi_{2,n})\in H^2(\mathbb{R})\times H^2(\mathbb{R})$ such that
\[\left\{
\begin{array}{c}
  -\varphi_{1,n}''+(3v_1^2-1)\varphi_{1,n}+\Lambda_n v_2^2 \varphi_{1,n}+2\Lambda_n v_1v_2\varphi_{2,n}=\mu_n \varphi_{1,n}, \\
    \\
  -\varphi_{2,n}''+(3v_2^2-1)\varphi_{2,n}+\Lambda_n v_1^2 \varphi_{2,n}+2\Lambda_n v_1v_2\varphi_{1,n}=\mu_n \varphi_{2,n}, \\
\end{array}\ \ z\in \mathbb{R},
\right.
\]
\[
\|\varphi_{1,n}\|_{L^\infty(\mathbb{R})}+\|\varphi_{2,n}\|_{L^\infty(\mathbb{R})}=1,
\]
and
\[
\int_{-\infty}^{\infty}\left(v_1' \varphi_{1,n}+v_2' \varphi_{2,n} \right)dz=0.
\]
Then, by absorbing $\mu_n \varphi_{i,n}$ in the term $(3v_i^2-1)\varphi_{i,n}$, $i=1,2$, the proof of Step 1 in Proposition \ref{proGenLinearQ} goes through to provide a contradiction.
\end{proof}

\section{Uniqueness of solutions: Proof of Theorem \ref{thmUniqNonSym}}\label{secUniq}

In this section, we will prove Theorem \ref{thmUniqNonSym}. The main task   will be to establish the uniqueness (modulo translations) of solutions to (\ref{eqEqGen})-(\ref{eqBdryGen}),  satisfying the natural monotonicity  property \begin{equation}\label{eqMonotFinal}
v_1'(z)>0,\ \ v_2'(z)<0,\ \ z\in \mathbb{R},
\end{equation}
for any $\Lambda$ in the range (\ref{eqHessian}). In particular, the latter monotonicity property is satisfied by stable solutions with positive components (see \cite[Thm. 3.1]{alamaARMA15}), and thus by minimizing ones. To the best of our knowledge, this type of uniqueness was not previously known, even in the case of minimizing solutions (see also \cite[Rem. 1.2]{alamaARMA15} and \cite[Rem. 4.8]{GL}). Once  the aforementioned  uniqueness property  is established, the corresponding assertion   of Theorem \ref{thmUniqNonSym}, where only one of the inequalities in (\ref{eqMonotFinal}) is assumed, will follow immediately  thanks to Lemma \ref{lemHalfMonot} below.
 %In passing, we note that the problem at hand does not seem to fit in the  framework of
%\cite{alikakos2006explicit}, where uniqueness and non-uniqueness issues for heteroclinic connections of related systems were studied  based on techniques from
%complex analysis. Moreover, let us
 We point out that system (\ref{eqEqGen}) is non cooperative, and that in the case $\Lambda <1$, uniqueness for a related problem follows from \cite{ANS}.

%More precisely, we have the following result.

The main result of this section is the following.

\begin{pro}\label{proUniqS}
	If $\Lambda>1$, there exists a unique solution (modulo translations) to (\ref{eqEqGen})-(\ref{eqBdryGen})-(\ref{eqMonotFinal}).
\end{pro}
\begin{proof}
	The proof is based on the key observation that uniqueness holds for $\Lambda=3$ (see  \cite{alamaARMA15} and the references therein) and a continuation argument.

Throughout this proof, we shall assume the 'pinning' condition:
 \begin{equation}\label{eqPining}
 v_1(0)=v_2(0).
 \end{equation}

	Firstly, and for future reference, we note that any solution of (\ref{eqEqGen})-(\ref{eqBdryGen}) with $\Lambda>1$ satisfies
	\begin{equation}\label{eqAlama}
	v_1^2+v_2^2< 1,\ \ z\in \mathbb{R},
	\end{equation}
	(see \cite[Thm. 2.4]{alamaARMA15}).
 %Actually, this bound holds for any solution of (\ref{eqEq}) that is defined on the whole real line. Moreover,   equality is not attained somewhere unless the %solution has identically constant components.

	We next claim that the following localization property holds: Let $\underline{\lambda}>1$ and $\varepsilon>0$, then there exists $M>0$ such that any solution of (\ref{eqEqGen})-(\ref{eqBdryGen})-(\ref{eqPining})
	with $\Lambda \geq \underline{\lambda}$ such that
	\begin{equation}\label{eqUniqNondec}
	v_1'(z)\geq 0,\ \ v_2'(z)\leq 0,\ \ z\in \mathbb{R},
	\end{equation}
	satisfies
	\begin{equation}\label{eqUniqClaimP}
	1-v_1(z)+v_2(z)<\varepsilon \ \ \textrm{for}\ \ z\geq M,
	\end{equation}
and the analogous relation for $z\leq -M$.
	Indeed, in view of the conservation of the Hamiltonian, it is enough to verify that, given $\epsilon>0$, there exists $L>0$ so that any such solution satisfies
	\[v'_1(z_0)-v_2'(z_0)<\epsilon\ \ \textrm{for\ some}\ \ z_0\in [0,L].
	\]
	If not, for any $L>0$, there would exist at least one such solution satisfying
	\[
	v_1'(z)-v_2'(z)\geq \epsilon \ \ \textrm{for}\ z\in [0,L],
	\]
	i.e.,
	\[
	v_1(L)-v_2(L)\geq \epsilon L,
	\]
	which is clearly not possible for large $L$ by virtue of (\ref{eqAlama}) and proves the claim.

In turn, similarly to \cite[Thm. 2.8]{bronsard1996three}, for any $1<\underline{\lambda}<\bar{\lambda}$, there exist constants $c,C>0$ such that any solution of (\ref{eqEqGen})-(\ref{eqBdryGen})-(\ref{eqPining})-(\ref{eqUniqNondec}) with $\Lambda \in \left[\underline{\lambda},\bar{\lambda}\right]$ satisfies
\begin{equation}\label{eqGui}
\sum_{i=1}^{2}\left\{|v_i''|+|v_i'|+|v_i-2+i| \right\}\leq Ce^{-cz},\ \ z\geq M,
\end{equation}
and the analogous estimate for $z\leq -M$.

	The previous observations have the following interesting implication: Let $\left(v_{1,n},v_{2,n}\right)$ be a sequence of solutions of (\ref{eqEqGen})-(\ref{eqBdryGen})-(\ref{eqMonotFinal})-(\ref{eqPining}) with $\Lambda=\Lambda_n$,
	such that
	\[
	v_{i,n}-v_{i,\infty}\to 0 \ \textrm{in}\ H^2(\mathbb{R}),\ i=1,2,\ \textrm{and}\ \Lambda_n \to \Lambda_\infty \in (1,\infty).
	\]
	Then, the limit $\left(v_{1,\infty},v_{2,\infty}\right)$ satisfies (\ref{eqEqGen})-(\ref{eqBdryGen})-(\ref{eqMonotFinal})-(\ref{eqPining}) with $\Lambda=\Lambda_\infty$. Indeed, since $C^1(\mathbb{R})$ is continuously imbedded into $H^2(\mathbb{R})$,  without loss of generality, it is enough to exclude the scenario where
	\begin{equation}\label{eqUniqScenario}
	v_{1,\infty}'(z_*)=0\ \ \textrm{for\ some}\ z_*\in \mathbb{R}.
	\end{equation}
	To this end, we note that $\varphi\equiv v_{1,\infty}'\geq 0$ satisfies
	\[
	-\varphi''+P(z)\varphi=-2\Lambda_\infty  v_{1,\infty}v_{2,\infty}v_{2,\infty}'\geq 0, \ \ z\in \mathbb{R},
	\]
	for some smooth function $P$. Thus, the above scenario (\ref{eqUniqScenario}) cannot happen, as it would violate a version of Hopf's boundary point lemma (see for example \cite[Thm. 2.8.4]{pucci2007maximum}).
	
	It follows from  \cite[Thm. 3.1]{alamaARMA15} that the linearized operator of (\ref{eqEqGen})  about a solution of (\ref{eqEqGen})-(\ref{eqBdryGen})-(\ref{eqMonotFinal})-(\ref{eqPining}), that is $\textbf{M}$ in (\ref{eqM}) with $H^2(\mathbb{R})\times H^2(\mathbb{R})$ as its domain, has a one-dimensional kernel spanned by $(v_1',v_2')$. We also note that this linear operator is self-adjoint in $L^2(\mathbb{R})\times L^2(\mathbb{R})$, and its continuous spectrum is contained in $[\Lambda^2-1,\infty)$ (see again \cite{alamaARMA15}).
 Therefore, by the variational Lyapunov-Schmidt procedure of Proposition \ref{proGenLinearQ} (in a regular perturbation setting) or a dynamical systems approach (see \cite{shatah2003orbits}), and the observation in the previous paragraph, we deduce the following: Each solution $(v_{1,\Lambda_0},v_{2,\Lambda_0})$ of (\ref{eqEqGen})-(\ref{eqBdryGen})-(\ref{eqMonotFinal})-(\ref{eqPining}), for some $\Lambda_0>1$, is contained in a locally unique and smooth for $|\Lambda-\Lambda_0|$ sufficiently small (with respect to variations from $(v_{1,\Lambda_0},v_{2,\Lambda_0})$ in the $H^2(\mathbb{R})\times H^2(\mathbb{R})$-norm) branch of solutions of (\ref{eqEqGen})-(\ref{eqBdryGen})-(\ref{eqMonotFinal})-(\ref{eqPining}). In fact, if the aforementioned local uniqueness property failed,  there would be such a solution with the associated linearized operator
 having a nontrivial element $(Z_1,Z_2)$ in its kernel
 such that $Z_1(0)=Z_2(0)$, which is impossible by the opposite sign of $v_1'$ and $v_2'$. We observe next that, thanks to (\ref{eqAlama}), (\ref{eqUniqClaimP}) and (\ref{eqGui}), any solution to (\ref{eqEqGen})-(\ref{eqBdryGen})-(\ref{eqMonotFinal})-(\ref{eqPining}) with $\Lambda \in [\underline{\lambda},\bar{\lambda}]$ satisfies
	\[
	\|v_1-v_{1,\Lambda_0} \|_{H^2(\mathbb{R})}+\|v_2-v_{2,\Lambda_0} \|_{H^2(\mathbb{R})}\leq C,
	\]
	where $C$ depends only on $\underline{\lambda},\bar{\lambda}>1$. Therefore, the aforementioned solution branch of (\ref{eqEqGen})-(\ref{eqBdryGen})-(\ref{eqMonotFinal})-(\ref{eqPining}) can be extended smoothly and uniquely for all $\Lambda>1$.

As was mentioned in the beginning of the proof, it has been  observed  that for $\Lambda=3$ there exists a unique solution of (\ref{eqEqGen})-(\ref{eqBdryGen})-(\ref{eqMonotFinal})-(\ref{eqPining}); in fact, this solution can be found explicitly. Indeed, letting $u\equiv v_1+v_2$ yields that
	\[
	u''+u-u^3=0,\ \ z\in \mathbb{R};\ \ u\to 1,\ \ z\to \pm \infty,
	\]
	that is $u\equiv 1$ and the aforementioned uniqueness follows at once. Hence, in the case where there was non-uniqueness of solutions to (\ref{eqEqGen})-(\ref{eqBdryGen})-(\ref{eqMonotFinal})-(\ref{eqPining}) for some $\Lambda>1$, we would have two of the previously described solution branches meeting at some $\Lambda_*>1$. However, this is not possible from the local uniqueness of the solution branches.
%as it would imply that the corresponding linearized operator at $\Lambda_*$ would have a nontrivial kernel.
	\end{proof}

\begin{rem}\label{remExHet}
We note that, starting the above continuation argument from $\Lambda=3$,   yields a non-variational  proof of   existence of the heteroclinic solution.
\end{rem}

Concerning the monotonicity condition (\ref{eqMonotFinal}), we have the following interesting  property which is motivated from \cite{farinaDCDS}, where the PDE version of system (\ref{eqBUsystem}) was considered and the concept of \emph{half-monotone} solutions was introduced.
\begin{lem}\label{lemHalfMonot}
Assume that $(v_1,v_2)$ is a solution to (\ref{eqEqGen})-(\ref{eqBdryGen}) with positive components. Then,
$v_1'>0$ implies that $v_2'<0$, and vice versa.
\end{lem}
\begin{proof}
Let us assume that \begin{equation}\label{eqhalfv1}
v_1'>0,\end{equation} (the other case can be treated completely analogously). We  note that,
in light of (\ref{eqAlama}), there exist sequences $\{z_n^\pm\}$ with $z_n^\pm\to \pm \infty$ such that \begin{equation}\label{eqNM} v_2'(z_n^\pm)<0\ \  \textrm{for}\ \   n\gg 1.\end{equation}
Then, in analogy to  \cite{farinaDCDS}, we let
\[
\sigma=\frac{v_2'}{v_2}.
\]
A simple computation, using (\ref{eqEqGen}) and (\ref{eqhalfv1}), gives that
\[
(v_2^2\sigma')'>2v_2^4 \sigma,\ \ z\in \mathbb{R}.
\]
Hence, in view of (\ref{eqNM}), we deduce by the maximum principle that $\sigma<0$ on $[z_n^-,z_n^+]$ for $n \gg 1$, and the lemma follows.
\end{proof}

\section{Asymptotic behaviour of the energy: Proof of Corollary \ref{corenergyExp}}\label{secEnergetic}
\begin{proof}[Proof of Corollary \ref{corenergyExp}]
In view of (\ref{eqmonotTHM}), Theorem \ref{thmUniqNonSym} and the discussion in the beginning of Section \ref{secUniq}, we infer that
the solution in Theorem \ref{thmMain} is the only minimizer (modulo translations) with positive components   to the  problem (\ref{eqsigmaL}). (Actually, a simple reflection argument, using $\left(|v_1|,|v_2| \right)$ as a competitor in the energy, yields that minimizers necessarily have positive components). Thus, in order to verify the assertion of Corollary \ref{corenergyExp}, it is enough to estimate the energy  of the aforementioned solution. For this purpose, a very helpful  observation is that, thanks to the conservation of the hamiltonian, we have
\[
E_\Lambda(v_1,v_2)=\int_{\mathbb{R}}^{}\left[(v_1')^2+(v_2')^2 \right]dz.
\]
 %\end{proof}
Firstly, making use of (\ref{v1UGrad}), we find that
\[
\int_{(\ln \Lambda)\Lambda^{-\frac{1}{4}}}^{\infty}(v_1')^2dz=\int_{(\ln \Lambda)\Lambda^{-\frac{1}{4}}+\psi_0^{-1}\kappa \Lambda^{-\frac{1}{4}}}^{\infty}
\left[U_1'(s)\right]^2ds+\mathcal{O}\left( (\ln \Lambda)\Lambda^{-\frac{3}{4}}\right)
\]
as $\Lambda \to \infty$. In turn, exploiting the fact that $U_1'(s)=\psi_0+\mathcal{O}(s^2)$ for $0\leq s \leq 1$, we can write
\[
\int_{(\ln \Lambda)\Lambda^{-\frac{1}{4}}}^{\infty}(v_1')^2dz=\int_{0}^{\infty}
\left[U_1'(s)\right]^2ds-\psi_0^2(\ln \Lambda)\Lambda^{-\frac{1}{4}}-\psi_0 \kappa \Lambda^{-\frac{1}{4}}+\mathcal{O}\left( (\ln \Lambda)^3\Lambda^{-\frac{3}{4}}\right)
\]
as $\Lambda \to \infty$.
On the other side, we obtain from (\ref{v1outt}) and (\ref{v1expdecay}) respectively  that
\[
\int_{-(\ln \Lambda)\Lambda^{-\frac{1}{4}}}^{(\ln \Lambda)\Lambda^{-\frac{1}{4}}}(v_1')^2dz=\Lambda^{-\frac{1}{4}}\int_{-(\ln \Lambda)}^{(\ln \Lambda)}(V'_1)^2dx+\mathcal{O}\left( (\ln \Lambda)^3\Lambda^{-\frac{3}{4}}\right)
\]
and
\[
\int_{-\infty}^{-(\ln \Lambda)\Lambda^{-\frac{1}{4}}}(v_1')^2dz=\mathcal{O}\left(\Lambda^{-\infty}\right)\ \ \textrm{as}\ \ \Lambda \to \infty.
\]
By adding the above three relations, we arrive at
\[\begin{array}{rcl}
    \int_{-\infty}^{\infty}(v_1')^2dz & = & \int_{0}^{\infty}
\left[U_1'(s)\right]^2ds+\Lambda^{-\frac{1}{4}}\left\{\int_{-(\ln \Lambda)}^{0}(V'_1)^2dx+\int_{0}^{(\ln \Lambda)}\left[(V'_1)^2-\psi_0^2\right]dx-\psi_0\kappa\right\} \\
      &   & +\mathcal{O}\left( (\ln \Lambda)^3\Lambda^{-\frac{3}{4}}\right).
  \end{array}
\]
Obviously, the righthand side increases negligibly if we replace the ends of integration $\pm \ln \Lambda$ with $\pm \infty$ respectively (keep in mind that $V_1$ is convex and that relation (\ref{eqV1V2asympt}) can be differentiated). The completely analogous relation holds for the second component. Therefore, observing that
\[
\int_{-\infty}^{0}(V'_1)^2dx+\int_{0}^{\infty}\left[(V'_1)^2-\psi_0^2\right]dx=\int_{-\infty}^{\infty}V'_1
\left(V'_1-\psi_0\right)dx,
\]
 it remains to verify that
\begin{equation}\label{eqpartition}\int_{-\infty}^{0}
\left[U_2'(s)\right]^2ds+
\int_{0}^{\infty}
\left[U_1'(s)\right]^2ds=\frac{2\sqrt{2}}{3},
\end{equation}
(recall also the symmetry property (\ref{eqV1V2sym})). This is indeed the case, as one can determine explicitly the value of each one of the above integrals (in fact, they are equal by  symmetry), thanks to the conservation of the hamiltonian of  problems (\ref{eqW1Gen}), (\ref{eqW2Gen}) (see for example \cite[Lem. 4.1]{GL}); we leave the details to the reader.
\end{proof}

\section*{Acknowledgments} This project has received funding from the CNRS, with an invited professor position for the second author. The second author wishes to thank the University of Versailles, where the first part of this paper was written, for the hospitality. He acknowledges  support from the ARISTEIA project DIKICOMA from Greece and from the European Union's Seventh Framework programme for  research and innovation under the Marie Sk\l{}odowska-Curie grant agreement No 609402-2020 researchers: Train to Move (T2M).  Moreover, he wishes to express his thanks to Prof. Terracini for   interesting and motivating discussions.
\bibliographystyle{plain}
\bibliography{biblioaacs}

\begin{thebibliography}{10}

\bibitem{ANS}
Amandine Aftalion, Benedetta Noris, and Christos Sourdis.
\newblock Thomas-{F}ermi approximation for coexisting two component
  {B}ose--{E}instein condensates and nonexistence of vortices for small
  rotation.
\newblock {\em Communications in Mathematical Physics}, 336(2):509--579, 2015.

\bibitem{AL}
Amandine Aftalion and Jimena Royo-Letelier.
\newblock A minimal interface problem arising from a two component
  {B}ose--{E}instein condensate via ${\Gamma}$-convergence.
\newblock {\em Calculus of Variations and Partial Differential Equations},
  52(1-2):165--197, 2015.

\bibitem{agudelo}
Oscar Agudelo and Andr{\'e}s Z{\'u}{\~n}iga.
\newblock A two-end family of solutions for the inhomogeneous {A}llen--{C}ahn
  equation in $\mathbb{R}^2$.
\newblock {\em Journal of Differential Equations}, 256(1):157--205, 2014.

\bibitem{alamaARMA15}
Stan Alama, Lia Bronsard, Andres Contreras, and Dmitry~E. Pelinovsky.
\newblock Domain walls in the coupled {G}ross--{P}itaevskii equations.
\newblock {\em Archive for Rational Mechanics and Analysis}, 215(2):579--610,
  2015.

\bibitem{alikakosFuscoIndiana}
Nicholas~D Alikakos and Giorgio Fusco.
\newblock On the connection problem for potentials with several global minima.
\newblock {\em Indiana Univ. Math. J.}, 57(04):1871--1906, 2008.

\bibitem{AO}
P.~Ao and S.T. Chui.
\newblock Binary {B}ose-{E}instein condensate mixtures in weakly and strongly
  segregated phases.
\newblock {\em Physical Review A}, 58(6):4836, 1998.

\bibitem{barankov}
RA~Barankov.
\newblock Boundary of two mixed {B}ose-{E}instein condensates.
\newblock {\em Physical Review A}, 66(1):013612, 2002.

\bibitem{berestycki-wei2012}
Henri Berestycki, Tai-Chia Lin, Juncheng Wei, and Chunyi Zhao.
\newblock On phase-separation models: asymptotics and qualitative properties.
\newblock {\em Archive for Rational Mechanics and Analysis}, 208(1):163--200,
  2013.

\bibitem{berestycki2}
Henri Berestycki, Susanna Terracini, Kelei Wang, and Juncheng Wei.
\newblock On entire solutions of an elliptic system modeling phase separations.
\newblock {\em Advances in Mathematics}, 243:102--126, 2013.

\bibitem{bronsard1996three}
Lia Bronsard, Changfeng Gui, and Michelle Schatzman.
\newblock A three-layered minimizer in $\mathbb{R}^2$ for a variational problem
  with a symmetric three-well potential.
\newblock {\em Communications on pure and applied mathematics}, 49(7):677--715,
  1996.

\bibitem{CaffLin2}
L.~A. Caffarelli and F.-H. Lin.
\newblock Singularly perturbed elliptic systems and multi-valued harmonic
  functions with free boundaries.
\newblock {\em J. Amer. Math. Soc.}, 21(3):847--862, 2008.

\bibitem{catelani}
Gianluigi Catelani and EA~Yuzbashyan.
\newblock Phase diagram, extended domain walls, and soft collective modes in a
  three-component fermionic superfluid.
\newblock {\em Physical Review A}, 78(3):033615, 2008.

\bibitem{ctv3}
M.~Conti, S.~Terracini, and G.~Verzini.
\newblock On a class of optimal partition problem related to the
  {F}u$\check{\text{c}}$\'ik spectrum and to the monotonicity formulae.
\newblock {\em Calc. Var. Partial Differential Equations}, 22(1):45--72, 2005.

\bibitem{dancer2011}
EN~Dancer, Kelei Wang, and Zhitao Zhang.
\newblock Uniform {H}{\"o}lder estimate for singularly perturbed parabolic
  systems of {B}ose--{E}instein condensates and competing species.
\newblock {\em Journal of Differential Equations}, 251(10):2737--2769, 2011.

\bibitem{del2010toda}
Manuel Del~Pino, Micha{\l} Kowalczyk, Frank Pacard, and Juncheng Wei.
\newblock The {T}oda system and multiple-end solutions of autonomous planar
  elliptic problems.
\newblock {\em Advances in Mathematics}, 224(4):1462--1516, 2010.

\bibitem{del2008giorgi}
Manuel del Pino and Juncheng Wei.
\newblock An introduction to the finite and infinite dimensional reduction
  method.
\newblock In Xingwang~Xu Fei~Han and Weiping Zhang, editors, {\em Geometric
  analysis around scalar curvatures}, pages 35--118. World Scientific,
  Singapore, 2016.

\bibitem{farinaDCDS}
Alberto Farina.
\newblock Some symmetry results for entire solutions of an elliptic system
  arising in phase separation.
\newblock {\em Discrete and Continuous Dynamical Systems}, 34(6):2505--2511,
  2014.

\bibitem{gallo}
Cl\'ement Gallo.
\newblock The ground state of two coupled {G}ross-{P}itaevskii equations in the
  {T}homas-{F}ermi limit.
\newblock {\em Journal de {M}ath\'{e}matiques {P}ures et {A}ppliqu\'{e}es},
  doi:10.1016/j.matpur.2016.02.001, 2016.

\bibitem{goldman2015phase}
Michael Goldman and Beno{\i}t Merlet.
\newblock Phase segregation for binary mixtures of {B}ose-{E}instein
  condensates.
\newblock {\em arXiv preprint arXiv:1505.07234}, 2015.

\bibitem{GL}
Michael Goldman and Jimena Royo-Letelier.
\newblock Sharp interface limit for two components {B}ose-{E}instein
  condensates.
\newblock {\em ESAIM: COCV}, 21(3):603--624, 2015.

\bibitem{grossi2008radial}
Massimo Grossi.
\newblock Radial solutions for the {B}rezis--{N}irenberg problem involving
  large nonlinearities.
\newblock {\em Journal of Functional Analysis}, 254(12):2995--3036, 2008.

\bibitem{herve1994etude}
Rose-Marie Herv{\'e} and Michel Herv{\'e}.
\newblock {\'E}tude qualitative des solutions r{\'e}elles d'une {\'e}quation
  diff{\'e}rentielle li{\'e}e {\`a} l'{\'e}quation de {G}inzburg-{L}andau.
\newblock In {\em Annales de l'IHP Analyse non lin{\'e}aire}, volume~11, pages
  427--440, 1994.

\bibitem{hislop}
Peter~D. Hislop and Israel~Michael Sigal.
\newblock {\em Introduction to spectral theory: With applications to
  {S}chr{\"o}dinger operators}.
\newblock Springer New York, 1996.

\bibitem{KT}
Kenichi Kasamatsu and Makoto Tsubota.
\newblock Vortex sheet in rotating two-component {B}ose-{E}instein condensates.
\newblock {\em Physical Review A}, 79(2):023606, 2009.

\bibitem{KTU}
Kenichi Kasamatsu, Makoto Tsubota, and Masahito Ueda.
\newblock Spin textures in rotating two-component {B}ose-{E}instein
  condensates.
\newblock {\em Physical Review A}, 71(4):043611, 2005.

\bibitem{kuhen}
Christian Kuehn.
\newblock {\em {M}ultiple {T}ime {S}cale {D}ynamics}, volume 191 of {\em
  {A}pplied {M}athematical {S}ciences}.
\newblock Springer, 2015.

\bibitem{lund2015existence}
Ross~G Lund, JM~Robbins, and Valeriy Slastikov.
\newblock Existence of travelling-wave solutions representing domain wall
  motion in a thin ferromagnetic nanowire.
\newblock {\em arXiv preprint arXiv:1512.06016}, 2015.

\bibitem{AM}
Peter Mason and Amandine Aftalion.
\newblock Classification of the ground states and topological defects in a
  rotating two-component {B}ose-{E}instein condensate.
\newblock {\em Physical Review A}, 84(3):033611, 2011.

\bibitem{mishonov}
Todor~M. Mishonov.
\newblock Comment on ''{I}nterface tension of {B}ose-{E}instein condensates" by
  {B}ert {V}an {S}chaeybroeck, {P}hys. {R}ev. {A} 78, 023624-9 (2008).
\newblock {\em arXiv preprint arXiv:1502.07171}, 2015.

\bibitem{NoTaTeVe}
B.~Noris, H.~Tavares, S.~Terracini, and G.~Verzini.
\newblock Uniform {H}{\"o}lder bounds for nonlinear {S}chr\"odinger systems
  with strong competition.
\newblock {\em Comm. Pure Appl. Math.}, 63(3):267--302, 2010.

\bibitem{pucci2007maximum}
Patrizia Pucci and James Serrin.
\newblock Maximum principles for elliptic partial differential equations.
\newblock {\em Handbook of Differential Equations--Stationary Partial
  Differential Equations}, 4:355--483, 2007.

\bibitem{schecter2010heteroclinic}
Stephen Schecter and Christos Sourdis.
\newblock Heteroclinic orbits in slow--fast hamiltonian systems with slow
  manifold bifurcations.
\newblock {\em Journal of Dynamics and Differential Equations}, 22(4):629--655,
  2010.

\bibitem{shatah2003orbits}
Jalal Shatah and Chongchun Zeng.
\newblock Orbits homoclinic to centre manifolds of conservative {PDE}s.
\newblock {\em Nonlinearity}, 16(2):591--614, 2003.

\bibitem{sz2}
Nicola Soave and Alessandro Zilio.
\newblock On phase separation in systems of coupled elliptic equations:
  asymptotic analysis and geometric aspects.
\newblock {\em arXiv preprint arXiv:1506.07779}, 2015.

\bibitem{sz}
Nicola Soave and Alessandro Zilio.
\newblock Uniform bounds for strongly competing systems: The optimal
  {L}ipschitz case.
\newblock {\em Archive for Rational Mechanics and Analysis}, pages 1--51, 2015.

\bibitem{son}
D.T. Son and M.A. Stephanov.
\newblock Domain walls of relative phase in two-component {B}ose-{E}instein
  condensates.
\newblock {\em Physical Review A}, 65(6):063621, 2002.

\bibitem{tavares2012regularity}
Hugo Tavares and Susanna Terracini.
\newblock Regularity of the nodal set of segregated critical configurations
  under a weak reflection law.
\newblock {\em Calculus of Variations and Partial Differential Equations},
  45(3-4):273--317, 2012.

\bibitem{vaninterface}
Bert Van~Schaeybroeck.
\newblock Interface tension of {B}ose-{E}instein condensates.
\newblock {\em Physical Review A}, 78(2):023624, 2008.

\bibitem{wang2015uniform}
Kelei Wang.
\newblock Uniform {L}ipschitz regularity of flat segregated interfaces in a
  singularly perturbed problem.
\newblock {\em arXiv preprint arXiv:1507.06104}, 2015.

\bibitem{WeWe1}
Juncheng Wei and Tobias Weth.
\newblock Asymptotic behaviour of solutions of planar elliptic systems with
  strong competition.
\newblock {\em Nonlinearity}, 21(2):305--317, 2008.

\bibitem{zhang2015singularities}
Shan Zhang and Zuhan Liu.
\newblock Singularities of the nodal set of segregated configurations.
\newblock {\em Calculus of Variations and Partial Differential Equations},
  54(2):2017--2037, 2015.

\bibitem{sternbergHeteroclinic2016}
Andr{\'e}s Z{\'u}{\~n}iga and Peter Sternberg.
\newblock On the heteroclinic connection problem for multi-well gradient
  systems.
\newblock {\em arXiv preprint arXiv:1604.03645}, 2016.

\end{thebibliography}
\end{document}